\documentclass[reqno,12pt]{article}
\title{One-dimensional long-range diffusion-limited aggregation I}
\author{Gideon Amir\thanks{University of Toronto}
  \and Omer Angel\thanks{University of British Columbia}
  \and Itai Benjamini\thanks{Weizmann Institute of Science}
  \and Gady Kozma\footnotemark[3]}

\date{October 2009}

\usepackage{pst-all,url}
\usepackage{graphicx, color}
\usepackage{varioref}
\usepackage{amsmath,amsthm,amssymb,amsfonts}
\usepackage{fullpage, sidecap}

\newtheorem{mainthm}{Theorem}
\newtheorem{thm}{Theorem}[section]
\newtheorem{lemma}[thm]{Lemma}
\newtheorem{coro}[thm]{Corollary}
\newtheorem{prop}[thm]{Proposition}
\newtheorem{claim}[thm]{Claim}
\newtheorem*{conj*}{Conjecture}
\theoremstyle{definition}
\newtheorem{defn}[thm]{Definition}
\newtheorem*{defn*}{Definition}
\theoremstyle{remark}
\newtheorem*{rem*}{Remark}

\newcommand{\thmref}[1]{Theorem~\ref{T:#1}}

\newcommand{\lemref}[1]{Lemma~\ref{L:#1}}
\newcommand{\corref}[1]{Corollary~\ref{C:#1}}
\newcommand{\clmref}[1]{Claim~\ref{C:#1}}

\newcommand{\ep}{\varepsilon}
\newcommand{\eps}{\varepsilon}
\renewcommand{\P}{\mathbb{P}}
\newcommand{\E}{\mathbb{E}}

\newcommand{\Z}{\mathbb{Z}}
\newcommand{\slt}{\preceq}

\newcommand{\F}{\mathcal{F}}

\DeclareMathOperator{\capa}{Cap}

\DeclareMathOperator{\diam}{diam}

\newcommand{\dist}{d}

\gdef\SetFigFontNFSS#1#2#3#4#5{} 

\begin{document}
\maketitle

\begin{abstract}
  We examine diffusion-limited aggregation generated by a random walk on
  $\Z$ with long jumps. We derive upper and lower bounds on the growth rate
  of the aggregate as a function of the number of moments a single step of
  the walk has. Under various regularity conditions on the tail of the step
  distribution, we prove that the diameter grows as $n^{\beta+o(1)}$, with
  an explicitly given $\beta$. The growth rate of the aggregate is shown to
  have three phase transitions, when the walk steps have finite third
  moment, finite variance, and, conjecturally, finite half moment.
\end{abstract}

\section{Introduction}

Start with a single seed particle fixed in space. Bring a second particle
from infinity, doing a random walk. Once it hits the first particle, freeze
it at the last place it visited before hitting the first particle. Bring a
third particle and freeze it when it hits the existing particles. Repeat,
and watch the aggregate grow. This process, known as diffusion-limited
aggregation, DLA for short, was suggested by physicists Witten and Sander
\cite{WS81} when the space is $\Z^2$. They ran simulations with several
thousand particles and discovered that a random fractal ensues. The
elegance of the model immediately caught the eyes of both physicists and
mathematicians.

However, very little has been proven about this model rigorously.
Eberz-Wagner \cite{EW} has some results about local statistics of the
aggregate. Kesten \cite{K87,K90,K91} proved non-trivial upper bounds for
the growth rate, but these do not demonstrate the fractal nature of the
model. Various simplified models have been suggested, but the fractal
nature of the aggregate is at best partially replicated. DLA on a cylinder
was shown to have a fingering phenomenon, when the base of the
cylinder mixes sufficiently rapidly \cite{BY08} (see also \cite{BH08}). In
the superficially similar \emph{internal} DLA, a process where the
particles start from $0$, walk on the aggregate and are glued at their point
of departure, the limiting shape is a ball \cite{LBG92,L95,LP,Sh}. A
similar phenomenon happens for the Richardson model, where the position of
the glued particle is picked from the uniform measure on the boundary of
the aggregate. Here the limit shape is some (unknown) convex shape which is a
far cry from being a fractal \cite{R73,NP95}. DLA on trees requires one to
adjust the parameters in order to get a ``fingering'' phenomenon
\cite{BPP97}. See also \cite{B93,K91,AAK01,NT08} and the fascinating
deterministic analog, the Hele-Shaw flow \cite{CM01,HM}.

In this paper we study one-dimensional \emph{long-range} DLA. The random
walk of the particles has unbounded long jumps. When such a long jump lands on
a site already in the aggregate the jump is not performed and the particle
is glued in its current position. Thus we deviate from the view of DLA as a
connected aggregate, but that is of course necessary to have an interesting
aggregate in one dimension. (As a particle system there are interesting
problems even in the connected one-dimensional case, see \cite{KS08}.)
One-dimensional long-range models have been studied for various questions
e.g.\ for percolation, the Ising model and others
\cite{S83,NS86,AN86,IN88,B04,BBY08}. Such models frequently exhibit
interesting phenomenology, reminiscent of the behaviour in $\Z^d$ but
different from it. In particular there is no canonical correspondence
between the dimension $d$ and the strength of the long-range interactions.

It is time to state our results (precise definitions will be given in
Chapter~\ref{sec:defs}). We say that a random variable $\xi$ has $\alpha$
moments if
\[
\alpha:=\sup\{a\ge 0 : \E |\xi|^a < \infty\}.
\]
A random walk $\{R_n\}$ has $\alpha$ moments if its step distribution does.
In particular if we have $\P(|R_1-R_0|=k)=k^{-1-\alpha+o(1)}$ then $R$ has
$\alpha$ moments. Our results focus on the effect of $\alpha$ on the growth
rate of the DLA generated by the random walk. A minimal form of our main
result is as follows.

\begin{mainthm}\label{T:all}
  Let $R$ be a symmetric random walk on $\Z$ with step distribution
  satisfying $\P(|R_1-R_0|=k) \sim c k^{-1-\alpha}$. Let $D_n$ be the
  diameter of the $n$ particle aggregate. Then almost surely:
  \begin{itemize}
  \item If $\alpha>3$, then $n-1 \le D_n \le Cn+o(n)$, where $C$ is a
    constant depending only on the random walk.
  \item If $2<\alpha\le 3$, then $D_n=n^{\beta+o(1)}$, where $\beta =
    \frac2{\alpha-1}$.
  \item If $1<\alpha<2$ then $D_n=n^{2+o(1)}$.
  \item If $\frac{1}{3}<\alpha<1$ then
    \[
    n^{\beta+o(1)} \le D_n \le n^{\beta'+o(1)}
    \]
    where $\beta=\max(2,\alpha^{-1})$ and $\beta' =
    \frac{2}{\alpha(2-\alpha)}$.
  \item If $0<\alpha<\frac{1}{3}$ then $D_n=n^{\beta+o(1)}$, where
    $\beta=\alpha^{-1}$.
  \end{itemize}
\end{mainthm}

\begin{figure}
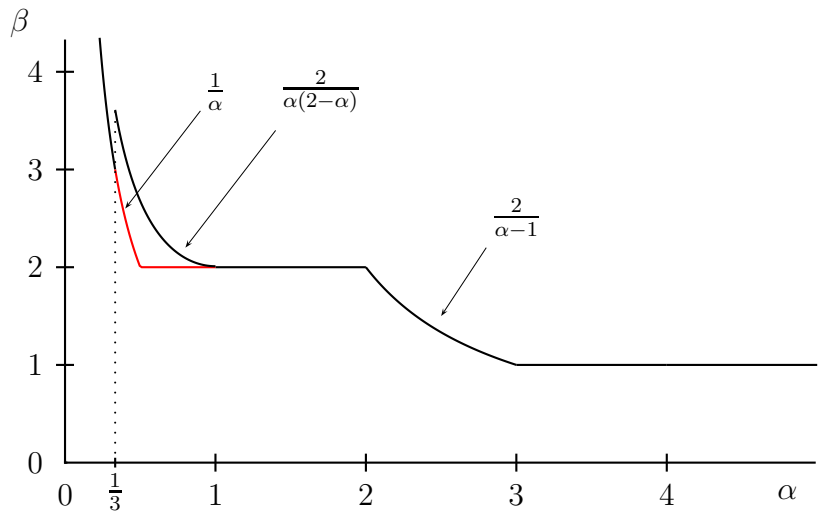

  \begin{center}
    \psset{xunit=20mm,yunit=13mm}
    \pspicture(-.25,-.4)(5.25,4.5)
    \psaxes(0,0)(4.99,4.33)
    \rput(-.3,4.5){$\beta$} \rput(4.8,-.3){$\alpha$}
    \psline(4,1)(5,1)
    \psline(3,1)(4,1)
    \psplot{2}{3}{2 x 1 sub div}
    \psline(1,2)(2,2)
    \psplot{.23}{.3333}{1 x div}
    \psplot[linecolor=red]{.3333}{1}{1 x div 2 max}
    \psplot{.3333}{1}{2 x div 2 x sub div .01 add}
    \psline[linestyle=dotted](.3333, 3.6)(.3333,.1)
    \psline(.3333,-.05)(.3333,.05)
    \rput(.3333,-.3){$\tfrac{1}{3}$}
    \psset{linewidth=0.01}
    \rput(1.7,3.8){$\frac2{\alpha(2-\alpha)}$} \psline{->}(1.4,3.4)(.8,2.2)
    \rput(1,3.8){$\frac1{\alpha}$} \psline{->}(0.9,3.6)(.4,2.6)
    \rput(3,2.5){$\frac2{\alpha-1}$} \psline{->}(2.8,2.2)(2.5,1.5)
    \endpspicture
    \caption{If the random walk $R$ has $\alpha$ finite moments, then the
      diameter of the resulting $n$-particle aggregate grows as $n^\beta$.
      For $\frac{1}{3}<\alpha<1$ our lower and upper bounds for $\beta$
      differ, and we believe the lower bound (in red) is correct.}
    \label{fig:graph}
  \end{center}
\end{figure}

Figure~\ref{fig:graph} depicts the various regimes described in
\thmref{all}. Not all of \thmref{all} is proved in this paper --- the cases
$\alpha<1$ are delegated to part II which is now being written. Let us
remark that this formulation is significantly weaker than our results below
for each regime. The theorems dealing with the various ranges of $\alpha$
apply to much more general random walks, and give more precise estimates on
the diameter $D_n$. The exact requirements and resulting estimates vary,
and the above formulation lies in their intersection. See the statement of
Theorems~\ref{T:third moment}, \ref{T:23_lower}, \ref{T:23_upper} and
\ref{T:12_both} throughout the text; and the results in part II. While our
results as stated do not cover the ``critical'' cases $\alpha=1,2$, the
reasons are mainly simplicity of presentation. The proofs, generally
speaking, can be extended to the boundary case with additional effort (and
sometimes with additional regularity conditions).

The most interesting feature of \thmref{all}, is of course the multiple
phase transitions --- as seen from Figure~\ref{fig:graph} --- at $3$, $2$
and at an (as yet) unknown place in $\left[\frac13,1\right]$. We feel
compelled to discuss them on a heuristic level. Before we consider the
transitions at 2 and 3 there is a point about the regime $\alpha>2$
that should be made.

When $\alpha>2$, $R$ has a finite second moment and the large scale
behaviour of $R$ is similar to that of the simple random walk on $\Z$. In
particular, the random walk has Brownian motion as its scaling limit. This
suggests that all walks with finite step variation will give rise to
similar DLA aggregates. As already stated, this is not the case. While the
Green function and the potential kernel of any such walk have linear
asymptotics (see Chapter~\ref{sec:finite variation}), the growth rate of
the DLA diameter can differ. The basic reason is that the walker is more
likely to discover new territory when making a large jump: A jump of size
$k$ takes the walker out of an interval where it has typically spent the
past $k^2$ steps. Such a jump is therefore roughly $k^2$ times more likely
to reach previously unvisited vertices. This causes large jumps to
contribute disproportionately to the aggregate growth. A similar effect is
exhibited by the ladder process corresponding to the random walk \cite[\S
  18]{S76}: A large jump is more likely to bring the walk to a new maximum,
and consequently the ladder steps have a thicker tail than the walk itself.

Throughout the regime $\alpha>2$ the process is directed: particles coming
from $+\infty$ have a bigger probability of hitting the right side and
particles coming from $-\infty$ will have a bigger probability of hitting
the left side. Further, each particle has probability bounded away from $0$
of hitting the extreme particle on its side and increase the aggregate's
diameter. The diameter can now be compared to a sum of i.i.d.\ variables
(though the increments are, of course, dependent) where if the expectation
is finite then the sum increases linearly whereas if the expectation is
infinite then the largest contribution dominates all the rest. The phase
transition at $3$ reflects a transition between a regime of incremental
additions and a regime of large jumps.

In the regime $2<\alpha<3$, a calculation shows that $\P(\Delta D_n>m)
\approx n m^{1-\alpha}$. When $m=n^{2/(\alpha-1)}$ this probability is
$\approx 1/n$ and there is probability bounded away from $0$ that at least
one such event occurs in the first $n$ particles. As explained above such
an event dominates all the rest and hence $D_n \approx n^{2/(\alpha-1)}$.
For this reason, the proof of the lower bound is the easier (the
calculation of the probability above, once justified, yields it
immediately). The upper bound requires to bound the contribution of the
smaller jumps, which turns out to be trickier and requires some insight
into the structure of the aggregate.

The next phase transition is at $\alpha=2$. This corresponds to the
transition from Gaussian behaviour of the random walk to stable behaviour:
for $\alpha<2$ the walk scales to an $\alpha$-stable process, and the Green
function grows like $n^{\alpha-1}$. In the recurrent stable regime,
$1<\alpha<2$, the calculation is quite similar to the one for the case
$2<\alpha<3$, but the result is that the additional contribution from the
fatter tail of the walk is exactly canceled by the slower growth of the
Green function and the growth of the aggregate is always $n^2$.

Let us first state what we believe is the true behaviour in the regime
$0<\alpha<1$, which is that the lower bound of \thmref{all} is sharp:

\begin{conj*}
  Let $R$ be a symmetric random walk on $\Z$ with step distribution
  satisfying $\P(|R_1-R_0|=k) \sim c k^{-1-\alpha}$ for some $0<\alpha<1$.
  Then $D_n = n^{\beta+o(1)}$ with $\beta = \max(2,\frac{1}{\alpha})$.
\end{conj*}

The reason for this conjecture will be discussed in more detail in part II,
but for now let us remark that we can prove this conjecture in the regime
$0<\alpha<\frac{1}{3}$. Moreover, a careful analysis of where the proof
fails in the regime $\frac{1}{3}<\alpha<\frac{1}{2}$ shows that for the
result not to hold requires the process to behave quite ridiculously (we
hope that the reader would forgive the unscientific language). Thus we have
a sound basis to believe that at $\alpha=\frac{1}{2}$, the aggregate grows
like $n^{2+o(1)}$. But this is exactly the growth rate at $\alpha=1$! It is
reasonable to believe that $\beta$ is decreasing as a function of $\alpha$
(though, again, we have no proof of that either), and hence the exponent
should be $2$ throughout the interval $[\frac{1}{2},1]$.

This conjecture raises two questions. The first: why is there a transition
at $\frac{1}{2}$? One may point at a certain transition in the behaviour of
a certain bound on the capacity of a fractal, but that is not much
different than saying ``because that is what the calculation shows''. But
the bigger question is: why is there no transition at $1$? After all, $1$
is the location of the most dramatic transition in our picture, the
transition between the recurrent and transient regimes (see
e.g.\ \cite[E8.2]{S76}). In the transient regime, one needs to modify the
definition of the process: we cannot simply have a particle ``coming from
infinity''. Instead one must condition on the particle ever hitting the
aggregate. Put differently, even though the processes at both sides of $1$
are different processes with only some kind of heuristic connection, they
still seem to grow at the same rate.

We remark that even if our conjecture is false and there is a phase
transition at $1$, it must be very weak. Indeed, the upper bound of
$n^{2/(\alpha(2-\alpha))}$ in \thmref{all} shows that $\beta(\alpha)$,
assuming it exists, must be differentiable at $1$ with, and $\beta'(1)=0$.
Thus despite the fundamental difference between the process for $\alpha<1$
and $\alpha>1$, the effect on the behaviour of $D_n$ is not so great. Let
us stress again that for all we know the growth rate of the aggregate at
$\frac{1}{3}<\alpha<1$ might be undefined, or depend on the particular
walk.

We should caution that there are difficulties in simulating the process to
get good numerical support for the conjecture. There are heuristic reasons
to believe (see Chapter~\ref{sec:Z3}) that the growth rate is not quite
$n^{\max(2,1/\alpha)}$, but that there are corrections which are at least
logarithmic in size.

While \thmref{all} and the bulk of our results describe only the behaviour
of the diameter of the aggregates, they give reason to believe that
rescaling the process might yield an interesting process. One could ask,
does the random set $A_n$ have a scaling limit? Note that there several
interpretations to this question. The scaling could be by some
deterministic factor, or by normalizing the set $A_n$ to the interval
$[0,1]$. There are also several topologies under which this question is
interesting, including the Hausdorff topology on subsets of $[0,1]$, and
weak convergence of the uniform measure on (the rescaled) $A_n$. A natural
topology to consider might be weak convergence of the rescaled harmonic
measure.

\medskip

Last but not least, let us discuss $A_\infty$, the infinite aggregate
defined as the union of the aggregates at all finite times. The natural
expectation is that the density of $A_\infty$ should reflect the growth
rate of $D_n$, at least to order of magnitude, that is if $D_n =
n^{\beta+o(1)}$ then
\begin{equation}
  \big| A_\infty\cap[-n,n] \big| = n^{1/\beta+o(1)}.  \label{eq:Ainfty}
\end{equation}
Indeed, in part III we give a proof of \eqref{eq:Ainfty} for the case
$\alpha>2$. When $\alpha > 3$ we show that the process has renewal times,
at which the subsequent growth of the aggregate is independent of the
structure of the aggregate. \eqref{eq:Ainfty} is a direct consequence. When
$2<\alpha<3$, these renewal times no longer exist. However, we show that it
is still hard for particles to penetrate deep into the aggregate, and
derive $\eqref{eq:Ainfty}$ in this case as well. The case $\alpha<2$ has
other difficulties and at present we are not ready to speculate on the
validity of \eqref{eq:Ainfty}. However, in Chapter~\ref{sec:Z3} we give an
example of a walk with $\alpha=0$, ``the $\Z^{3}$ restricted walk'' for
which, despite the fact that $D_n$ grows super-exponentially,
$A_\infty=\Z$. We do not know if such examples exist for $0<\alpha<2$, as
the construction we use is somewhat special.

\paragraph{Roadmap}

In Chapter~\ref{sec:defs} we derive a general formula for the gluing
measure in the recurrent case. This chapter is a prerequisite for the rest
of the paper. We then highly recommend reading Chapter~\ref{sec:Z2} in
which we analyze one specific case, the $\Z^2$ restricted walk (this walk
has $\alpha=1$). The proof in this case is much easier than the other
cases, and does not require any knowledge of stable random variables. The
next chapters are arranged by $\alpha$; Chapter~\ref{sec:finite third} for
$\alpha>3$, Chapter~\ref{sec:finite variation} for $2<\alpha<3$,
Chapter~\ref{sec:12} for $1<\alpha<2$. Finally, Chapter~\ref{sec:Z3}
describes the aforementioned example with $\alpha=0$.

A nontrivial portion of the paper (and of part II) is dedicated to discrete
potential theory, both general and that of stable walks
(e.g.\ Lemma~\ref{L:alpha12_tools}). We expected to find many of these
results in standard references, and did not.


\subsection*{Acknowledgements}

We thank Vlada Limic for useful suggestions. While working on this project,
GA was supported by the Weizmann Institute and the University of Toronto;
OA was supported by the universities of Paris at Orsay, British Columbia
and Toronto, by Chateaubriand and Rothschild Fellowships and by NSERC; IB
was supported by the Renee and Jay Weiss Chair; GK was supported by, in
chronological order, the Weizmann Institute of Science (Charles Clore
fund), Tel Aviv University, the Institute of Advanced Science (Oswald
Veblen fund and NSF grant DMS-0111298) and the Sieff prize. We would like
to thank all these institutions and funds.

\section{Preliminaries}
\label{sec:defs}

\subsection{Notations}

For a subset $A\subset \Z$ we will denote by $\diam A$ the diameter of $A$,
namely $\max A - \min A$. For $x\in\Z$ we will denote by $\dist(x,A)$ the
point-to-set distance, namely $\min_{y\in A}|x-y|$. Throughout we let $A_n$
be the $n$ point aggregate, and denote $D_n=\diam A_n$, $\Delta D_n =
D_{n+1}-D_n$. Let $\F_n$ be the minimal $\sigma$-field generated by
$A_0,\dotsc,A_n$.

We denote a single step of the random walk by $\xi$, and the random walk
itself by $R=(R_0,R_1,\dotsc)$. We denote by $\P_x$ the probability measure
of the random walk started at $x$. The transition probabilities of the
random walk are denoted by $p_{x,y} = \P(\xi=y-x)$. For a given set $A$,
define
\[
p(x,A)=\sum_{a\in A}p_{x,a}.
\]
We denote by $T_A$ be the hitting time of $A$, defined as
\[
T_A = \min\{n>0 \text{ s.t. } R_n\in A\}.
\]
Note that $T_A>0$ even if the random walks starts in $A$. For a set $\{x\}$
with a single member we also write $T_x$ for $T_{\{x\}}$. Denote by
$g(x,y)$ the Green function of $R$ defined by
\[
g(x,y)=\sum_{n=0}^\infty \P_x(R_n=y)\,.
\]
If $A\subset\Z$ then we define the relative Green function (a.k.a.\ the
Green function for the walk killed on $A$) by
\[
g_A(x,y) =\sum_{n=0}^\infty \P_x(R_n=y,\,T_A>n)
\]
and the hitting measure by
\[
H_A(x,a)=\begin{cases}
\P_x(R_{T_A}=a) & x\not\in A\\
\delta_{x,a} & x\in A
\end{cases}
\qquad H_A(\pm\infty,a)=\lim_{x\to\pm\infty}H_A(x,a)
\]
by \cite[T30.1]{S76} the limit on the right-hand side exists for any
aperiodic random walk.

By $C$ and $c$ we denote constants which depend only on the law of $\xi$
but not on any other parameter involved. The same holds for the constants
hidden in the $o(\cdot)$ notation, except when it is used in estimates for
$D_n$ (as in \thmref{all} above and other results below) where the factor
$o(\cdot)$ is random (This should always be clear from the context).
Generally $C$ and $c$ might take different values at different places, even
within the same formula. $C$ will usually pertain to constants which are
``big enough'' and $c$ to constants which are ``small enough''.

$X\lesssim Y$ denotes that $X<C Y$. By $X\approx Y$ we mean $cX<Y<CX$ (that
is, $X\lesssim Y \lesssim X$). By $X\sim Y$ we mean $X=(1+o(1))Y$. By
$X\asymp Y$ we mean that $X/Y$ is a 
slowly varying function --- see Chapter~\ref{sec:12} for details. $\lfloor
x \rfloor$ denotes the integer value of $x$.

\subsection{Gluing measures}

Let $R$ be a recurrent aperiodic random walk on $\Z$. (Recall that a random
walk is aperiodic if for any $x,y\in\Z$ there exists some $n$ such that
$\P_x(R_n=y) \neq 0$.) We will assume implicitly throughout the paper that
all our random walks are aperiodic. Let $A\subset\Z$ be some finite set,
and $T_A$ the (a.s.\ finite) hitting time of $A$ by $R$. We would like to
define the measure $\P_\infty = \lim_{y\to\infty} \P_y$. However, care must
be taken here, since the limit is not a probability measure using the
natural $\sigma$-algebra of the random walks, and the laws of natural
quantities such as $R_n$ or $T_A$ do not have a limit. However, the law of
$R_{T_A}$ --- the point at which $A$ is hit --- does have a limit, as do
the probabilities of events like $\{T_a<T_b\}$.

We define the measure $\P_{+\infty}$, depending implicitly on $A$,
as follows. This measure is supported on paths $\{\gamma_i\}_{i\leq
0}$, i.e.\ paths with no beginning but a last step. It is defined as
the limit as $y\to\infty$ of the law of $\{R_{T_A+i}\}_{i\leq 0}$.
Informally, $\P_{+\infty}$ is interpreted as the random walk started
at $+\infty$, and stopped when it hits $A$. Clearly it is supported
on paths in $\Z\setminus A$, except for $R_0\in A$. The measure
$\P_{-\infty}$ is defined similarly using $y\to-\infty$. We define
the measure $\P_\infty = \frac12(\P_{+\infty} + \P_{-\infty})$.
Finally, let
\[
\mu(x,a) = \mu(x,a;A) = \P_\infty(R_{-1}=x,R_0=a)
\]
be the probability that the random walk hits $A$ by making a step from $x$
to $a$.

\begin{lemma}\label{L:murecurr}
  For any recurrent random walk and any finite set $A$, the limits
  $\P_{\pm\infty}$ exist and are probability measures. For any $x_0\in A$
  and $x_{-1},\ldots,x_{-n}\notin A$
  \begin{equation}
    \P_{\pm\infty}(R_i=x_i \text{ for } -n\le i\le 0) =
    \frac{\P_{\pm\infty} (T_{x_{-n}} < T_A)} {\P_{x_{-n}}(T_A <
      T_{x_{-n}})} \prod_{i=-n}^{-1} p_{x_i,x_{i+1}}
    \label{eq:Ppmrecurr}
  \end{equation}
  and in particular
  \begin{equation}
    \mu(x,a) = \mu(x,a;A) =
    \frac{p_{x,a} \P_{\infty} (T_x < T_A)} {\P_x(T_A < T_x)}.
    \label{eq:murecurr}
  \end{equation}
\end{lemma}

\begin{proof}
  Fix a starting point $y$ and denote
  \[
  \P_y(x_0,x_{-1},\ldots,x_{-n}) =
  \P_y(R_{T_A-i} = x_i \text{ for } -n\le i\le 0) \,.
  \]
  (Where $R_{T_A-k}$ is undefined if the walk hits $A$ in less than $k$
  steps.) For clarity, write $z=x_{-n}$. Now, in order for the event on the
  right-hand side to happen, the walk must first hit $z$, which happens
  with probability $\P_y(T_z < T_A)$. By the strong Markov property at
  $T_z$, with probability $\P_z(T_A < T_z)$ the walk will not hit $A$
  before its next return to $z$. Thus the expected number of visits to $z$
  before $T_A$ is
  \[
  \frac{\P_y(T_z < T_A)} {\P_z(T_A < T_z)}.
  \]
  At each of these visits there is probability $\prod_{i=-n}^{-1}
  p_{x_i,x_{i+1}}$ of making the prescribed sequence of jumps ending at
  $x_0\in A$. Since the walk is stopped once such a sequence of jumps is
  made, the events of making these jumps after the $i$'th visit to $z$ are
  disjoint (for different $i$'s). Hence
  \[
  \P_y(x_0,x_{-1},\ldots,x_{-n}) =
  \frac{\P_y(T_z<T_A)}{\P_z(T_A<T_z)}\prod_{i=-n}^{-1} p_{x_i,x_{i+1}} \,.
  \]

  Thus to see that $\lim_{y\to\pm\infty} \P_y(x_0,x_{-1},\ldots,x_{-n})$
  exists, it suffices to show that $\lim \P_y(T_z<T_A)$ exists. Recall that
  the harmonic measure from infinity on a finite set $A$ is defined by
  \[
  H_A(\pm\infty,a) = \lim_{y\to\pm\infty} H_A(y,a) =
  \lim_{y\to\pm\infty} \P_y(R_{T_A}=a) \, .
  \]
  By \cite[T30.1]{S76}, this limit always exists. Note that
  \[
  \P_y(T_z<T_A) = \P_y(R_{T_{A\cup\{z\}}}=z) = H_{A\cup\{z\}}(y,z) \,.
  \]
  Existence of $\lim_{y\to\pm\infty} \P_y(x_0,x_{-1},\ldots,x_{-n})$
  follows.

  It remains to show that the limit is a probability measure i.e.\ that
  \[
  \sum_{\substack{x_0\in A\\ x_{-1},\ldots x_{-n}\notin A}}
  \P_{\pm\infty}(x_0,x_{-1},\ldots,x_{-n}) = 1.
  \]
  For any finite starting point $y$ this sum is $1$ by recurrence. The
  problem is that as $y\to\pm\infty$, the walk might be have a high
  probability of hitting $A$ by a large jump, so that for some $i$, the law
  of $x_i$ is not tight as $y\to\infty$. However if we show that the law of
  $x_{-n}$ is tight, then $\lim_{y\to\pm\infty}
  \P_y(x_0,x_{-1},\ldots,x_{-n})$ will be a probability measure.

  \begin{claim}\label{C:tight}
    For any finite $A\subset Z$,
    \[
    \lim_{m\to\infty}H_{[-m,m]}(\pm\infty,A)=0 \,.
    \]
  \end{claim}

  \begin{proof}
    For clarity, we use $H(x,y;A)$ in place of with $H_A(x,y)$. It suffices
    to prove the claim for a singleton $A=\{a\}$. We may assume $a\ge 0$.
    In this case we write
    \begin{align*}
      &&1&=\sum_{|x|\le m} H\big(\pm\!\infty,x; [-m,m]\big) \\
      &\mbox{by monotonicity} &&
      \ge \sum_{|x|\le m} H\big(\pm\!\infty,x; [x-2a-2m,x+2m]\big) \\
      &\mbox{by translation invariance} &&
      =\sum_{|x|\le m} H\big(\pm\!\infty,a; [-a-2m,a+2m]\big)\,.
    \end{align*}
    Hence $H\big(\pm\!\infty,a; [-(a+2m),a+2m]\big) \le 1/(2m+1)$.
  \end{proof}

  Returning to the proof of \lemref{murecurr}, fix $\eps>0$. For any finite
  set $A$ and any $n$ we can pick a sequence of finite intervals $A\subset
  I_0 \subset I_1\subset\ldots \subset I_n$ so that for any $k<n$ and any
  $y\notin I_k$, the probability from $y$ of hitting $I_k$ at a point of
  $I_{k-1}$ is at most $\eps/n$. We get
  \[
  \P_y(T_{I_n}<T_A-n) <\eps\qquad\forall y\not\in I_n
  \]
  and therefore by the strong Markov property at the stopping time $T_{I_n}$,
  \[
  \P_y(|R_{T_A-n}|>M) < \E_y\P_{T_{I_n}}(|R_{T_A-n}|>M)+\eps\qquad
  \forall M\;\forall y\not\in I_n\,.
  \]
  Now, the law of
  $R_{T_A-n}$ w.r.t.\ any starting point $I_n$ is tight (since these are
  just $|I_n|$ distributions). Hence we get
  \[
  \lim_{M\to\infty}\max_{y\in\Z}\P_y(|R_{T_A-n}|>M)=0\,,
  \]
  which is the required tightness.
\end{proof}

\begin{defn}\label{def:DLA}
  Let $R$ be a random walk on $\Z$. The \emph{DLA process with respect to
    $R$} is a sequence of random sets $A_0=\{0\} \subset A_1 \subset
  \cdots$ such that for any $A\subset \Z$, and $x\in\Z\setminus A$ and any
  $n>0$,
  \begin{equation}
  \P(A_{n+1} = A\cup\{x\}\,|\,A_n=A) = \sum_{a\in A} \mu(x,a;A)
  \label{eq:DLA_def}
  \end{equation}
  where $\mu$ is defined by \eqref{eq:murecurr}.
\end{defn}

When $\xi$ has infinite variance, $\P_{+\infty}(T_x < T_A) = \P_{-\infty}
(T_x < T_A)$ for any $x$ and $A$ and indeed $\P_{+\infty}=\P_{-\infty}$
\cite[T30.1 (1)]{S76}, but otherwise $\P_{+\infty}$ and $\P_{-\infty}$
differ. It is possible to define the DLA using walks that start only at
$+\infty$ or $-\infty$. This leads to minor variations on our results, and
the proofs remain valid with minimal modification.

Since by the RHS of \eqref{eq:DLA_def} the probability of adding a point
$x$ to $A_n$ can be interpreted as a measure over infinite paths ending at
$A_n$, we will say that $x$ is ``glued'' to $A_n$ at $a$ if the last two
steps of the path of the added particle are $x$ and $a$. The measure $\mu$
is thus called the ``gluing'' measure.

\section{The restricted $\Z^2$ walk}
\label{sec:Z2}

In the chapter we discuss a special random walk on $\Z$ resulting from an
embedding of $\Z$ as a sub-group of $\Z^2$, say as the diagonal
$\{(x,x)\}_{x\in \Z}$. Consider the sequence of vertices of $\Z$ visited by
a simple random walk on $\Z^2$, i.e.\ the restriction of the random walk to
$\Z$. This sequence of vertices forms a random walk on $\Z$. It is well
known that this walk has $\alpha=1$, and more precisely that the steps of
this random walk have approximately the Cauchy distribution, i.e.
$\P(\xi=k) \sim c|k|^{-2}$ (for the special case of the diagonal
embedding, there is even a precise formula \cite[E8.3]{S76}
$\P(\xi=0)=1-\frac{2}{\pi} ,\ \P(\xi=k)= \frac{2}{\pi(4k^2-1)}$ but we do
not use this extra precision). The fact that $\Z^2$ is recurrent
immediately implies that the restricted $\Z^2$ walk is recurrent as well,
hence we may consider the DLA formed by this walk.

While the restricted $\Z^2$ walk is a very special example, its study has
merit. The proofs are simpler, but the general ideas are the basis for the
proofs in more general cases. The reason the proofs are simpler is the vast
and very precise knowledge concerning the behaviour of the simple random
walk in $\Z^2$. This allows us to get sharp bounds for various quantities.
Addition of a vertex to the DLA in $\Z$ may be studied by examining a
simple random walk on $\Z^2$ and considering the last visit to $\Z$ before
hitting $A$.

\begin{thm}\label{T:Z2_lower}
  Consider the DLA generated by the $\Z^2$ restricted walk. For some $c>0$
  we have almost surely
  \[
  \liminf \frac{D_n \log\log n}{n^2} > c  \qquad \text{and} \qquad
  \limsup \frac{D_n}{n^2} = \infty.
  \]
\end{thm}

\thmref{Z2_upper} below gives a matching upper bound for $D_n$, up to
logarithmic factors. Together we find that the diameter grows essentially
quadratically. It is reasonable to believe that $\{n^{-2} D_{nt}\}_t$
converges to some random process, though it is not even proven that even
the law of $n^{-2} D_n$ converges. We argue as follows: if $D_n$ is small
then there is some probability that $D_{n+1}$ is large. We estimate this
probability for a suitable threshold for being ``large''. We then bound
this probability uniformly in $A_n$. By Borel-Cantelli it follows that
$D_n$ is large for infinitely many $n$. To make this precise, suppose
$D_n>m$. Then $D_{n+1}>m$ as well. On the other hand, for any set $A_n$
with $D_n\le m$ we have the following:

\begin{lemma}\label{L:Z2_delta_large}
  In the DLA generated by the $\Z^2$ restricted walk, there is a constant
  $c>0$ so that for any $A$ and $m$ with $\diam(A) < m$ we have
  \[
    \P(\Delta D_n>m | A_n=A) \ge \frac{c n}{m}.
  \]
\end{lemma}

\begin{figure}
 \begin{center}
   \includegraphics{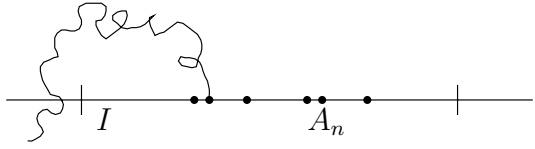}
   \begin{picture}(0,0)(0,0)
     \put(-90,5){$A_n$}
     \put(-170,5){$I$}
   \end{picture}
  \caption{\label{cap:Z2} The event that the random walk hits $I$ in
    $A_n$.}
\end{center}
\end{figure}

\begin{proof}
  Define the interval $I\subset\Z$ to be the $m$-neighborhood of $A_n$.
  (This is an interval since $\diam(A) < m$.) Consider a random walk in
  $\Z^2$ used for a DLA step, and consider the first time it hits $I$. If
  it hits $I$ at one of the points of $A_n$, then the previous visit to
  $\Z$ must have been at distance more than $m$ from $A_n$. See figure
  \ref{cap:Z2}. In that case a far point is added to $A_n$ and $\Delta
  D_n>m$. Hence
  \[
    \P(\Delta D_n>m) \ge H_I(\infty,A_n),
  \]
  where $H_I$ is the harmonic measure from infinity on $I$ for a random
  walk in $\Z^2$. We now use the well known fact\footnote{This follows from
    translation invariance and the Skorokhod invariance principle: By
    translation invariance it suffices to show this for $I=[-m,m]$. Let
    $J=[-3m,3m]$. By the invariance principle, $H_J(I)$ is bounded below by
    some constant $c$ independent of $m$. This implies that $H_J(x) > c/m$
    for some point $x\in I$. Translation invariance and monotonicity of the
    harmonic measure now imply (as in the proof of \clmref{tight}) that for
    any point $y\in I$ $H_I(y) = H_{I+(x-y)}(x) \geq H_J(x) \geq c/m$ as
    required.} that the harmonic measure satisfies the bound $H_I(x) \ge
  c/|I|$ for some universal $c$ and any $x\in I$ (near the ends of the
  interval the harmonic measure is much larger).

  It follows that
  \[
    H_I(A_n) \ge \frac{c}{|I|} |A_n| = \frac{cn}{D_n+2m+1} \ge
    \frac{cn}{3m}
  \]
  as required.
\end{proof}

Recall that $\F_n$ is the $\sigma$-algebra spanned by $A_1,\dots,A_n$. In
preparation for the treatment of more general walks, we prove the following
lemma. \thmref{Z2_lower} follows by applying the following to $M_n=D_n$
with $\beta=2$.

\begin{lemma}\label{L:general_lower}
  Let $\{M_n\}$ be a non-decreasing sequence adapted to a filtration
  $\{\F_n\}$, and suppose $\P(M_{n+1} > m | \F_n) \ge c_1 n m^{-2/\beta}$ for
  some $c_1>0$ and all $m,n>0$. Then there is some deterministic value
  $K>0$ such that a.s.
  \begin{align*}
    \limsup n^{-\beta} M_n &= \infty & \text{and} & &
    \liminf n^{-\beta}\left(\log\log n\right)^{\beta/2} M_n & > K.
  \end{align*}
\end{lemma}

\begin{proof}
  Take $m=a n^{\beta}$. By the conditions of the lemma
  \[
  \P(M_{n+1} \ge a n^\beta \,|\, \F_n) \ge c_1 n(a n^\beta)^{-2/\beta}
  \ge c n^{-1},
  \]
  uniformly in $\F_n$. Consequently $M_n\ge a n^{\beta}$ infinitely often.

  To estimate $\liminf \frac{M_n (\log\log n)^{\beta/2}}{n^\beta}$, take $m
  = \frac{c_1^{\beta/2} n^\beta}{(4\log\log n)^{\beta/2}}$. It follows that
  \[
  \P\left(M_{n+1} \ge \frac{c_1^{\beta/2} n^\beta}{(4\log\log n)^{\beta/2}}
  \,\Big|\, \F_n\right) \ge \frac{4\log\log n}{n}.
  \]
  Consequently, the probability that
  $M_{n+1} \leq \frac{c_1^{\beta/2} n^\beta}
  {(4\log\log n)^{\beta/2}}$ for all $n \in [N,2N)$ is at most
  \[
    \prod_{n=N}^{2N-1} \left( 1-\frac{4\log\log n}{n} \right)
    \le \left( 1-\frac{4\log\log N}{2N} \right)^N
    \le e^{ -2\log\log N } = \frac{1}{\log^2 N}.
  \]
  Considering only $N$ of the form $2^k$, we find that a.s.\ for all large
  $k$ there is a some $n_k \in [2^k,2^{k+1})$ such that $M_{n_k} \ge
  \frac{c_1^{\beta/2} n_k^2}{(4 \log\log n_k)^{\beta/2}}$. For any other
  $n$ we argue, using the monotonicity of the sequence $\{M_n\}$, that if
  $n\in[2^{k+1},2^{k+2}]$ then
  \[
  M_n \geq M_{n_k} \geq \frac{c_1^{\beta/2} n_k^\beta}{(4\log\log
    n_k)^{\beta/2}}
  \geq \frac{ c_1^{\beta/2} (\frac{n}{4})^\beta}{(4\log\log n)^{\beta/2}}.
  \]
  Thus $\liminf \frac{M_n (\log\log n)^{\beta/2}}{n^\beta} >
  \frac{c_1^{\beta/2}}{2^{3\beta}}$ a.s.
\end{proof}

\thmref{Z2_lower} now follows from Lemmas~\ref{L:Z2_delta_large} and
\ref{L:general_lower}. As promised, we have a matching upper bound, up to
logarithmic factors:

\begin{thm}\label{T:Z2_upper}
  For the DLA generated by the $\Z^2$ restricted walk, a.s.\ for any
  $\eps>0$ and all large enough $n$, $D_n\le n^2 (\log n)^{3+o(1)}$.
\end{thm}

We begin by bounding the probability of a large increment in $D_n$:

\begin{lemma}\label{L:Z2_delta_small}
  In the DLA generated by the $\Z^2$ restricted walk,
  \[
    \P(\Delta D_n > m | \F_n) \lesssim \frac{n \log m}{m}.
  \]
\end{lemma}

\begin{proof}
  We use the asymptotics
  \[
    \P_x(T_A<T_x) \approx \frac1{\log \dist(x,A)},
  \]
  (assuming $\dist(x,A) > 1$). This follows from asymptotics of the 2
  dimensional random walk: The probability of reaching (in $\Z^2$) distance
  $\frac12 \dist(x,A)$ before returning to $x$ is of order $\log^{-1}
  \dist(x,A)$ (see e.g.\ \cite[Lemma 9]{BKYY}). On this event the
  probability of hitting $A$ before returning to $x$ is bounded away
  from $0$ (even if $A$ contains a single point).

  Since $\P_\infty(T_x<T_A)\le 1$, the gluing formula
  (\ref{eq:murecurr}) implies
  \[
    \mu(x,a) \lesssim p_{x,a} \log \dist(x,A).
  \]
  Summing over all $x$ with $\dist(x,A)>m$ we get
  \begin{align*}
   \P(\Delta D_n > m|\F_n)
   &\lesssim 
      \sum_{\substack{a\in A\\x:\dist(x,A)>m}} p_{x,a} \log\dist(x,A) \\
   &\le \sum_{\substack{a\in A\\|x-a|>m}} p_{x,a} \log\dist(x,a) \\
   &\approx \frac{n \log m}{m}.
  \end{align*}
  The last estimate comes from the fact that the restricted $\Z^2$ walk
  satisfies $\P(\xi=k) \approx c|k|^{-2}$.
\end{proof}

This allows us to bound the probability of $D_n$ being large: We will also
need the following lemma, which translates upper bounds on the probability
of making large jumps into upper bounds on $D_n$

\begin{lemma}\label{L:general_upper}
  If
  \[
    \P(\Delta D_n > m | \F_n) \lesssim \frac{n \log m}{m}.
  \]
  Then for all $\gamma$,
  \[
    \P(D_N \ge \gamma N^2 \log^2 N) \lesssim \frac{1}{\gamma}.
  \]
  If one has the weaker $\P(\Delta D_n>m | \F_n) \leq n m^{-1+o(1)}$
  then one gets $\P(D_N\ge\gamma \phi(N))\lesssim 1/\gamma$ for some
  deterministic $\phi(N)=N^{2+o(1)}$.
\end{lemma}

\begin{proof}
  Fix $M>N^2$, and set for $0\le k \le \log_2 M$
  \[
  B_k = \left\{ n \le N : \Delta D_n \in
    \left(\frac{M}{2^{k+1}},\frac{M}{2^k} \right] \right\}
  \]
  and
  \[
  B_{-\infty} = \left\{ n \le N : \Delta D_n > M \right\}.
  \]
  We argue that with high probability the contribution to $D_n$ from
  increments in each of the $B_k$'s is at most $M$. The event $\{D_n \geq
  M(2+\log_2 M)\}$ is a subset of the event
  \[
  \exists k\in\{-\infty,0,\dots,\log_2 M\} \text{ such that } \sum_{n\in B_k}
  \Delta D_n > M.
  \]
  By a union bound,
  \begin{equation}\label{eq:Z2_union}
    \P\big(D_N > M(2+\log M)\big)
    \leq \P\big(B_{-\infty} \neq \emptyset\big)
        + \sum_k \P\left( |B_k| > 2^k \right).
  \end{equation}
  By the conditions of the lemma,
  \[
    \P(B_{-\infty} \neq \emptyset) \le N \cdot C \frac{N \log M}{M}.
  \]
  Using also the bound $\binom{N}{a} \le \left(\frac{eN}{a}\right)^a$,
  \begin{align*}
    \P\left( |B_k| > 2^{k} \right)
    &\leq \binom{N}{2^k} \left( CN 2^k \frac{\log M}{M} \right)^{2^k} \\
    &\leq \left( CN^2 \frac{\log M}{M} \right)^{2^k}.
  \end{align*}
  Setting $M=\gamma N^2 \log N$ and using the above bounds in
  \eqref{eq:Z2_union} we get
  \[
    \P(D_N > M(\log_2 M+2)) \le \frac{C}{\gamma}
    + \sum_{k\geq 0} \left(\frac{C_1}{\gamma}\right)^{2^k}.
  \]
  Clearly we may assume $\gamma$ is sufficiently large (by enlarging the
  constant implicit in the $\lesssim$ in the statement of the lemma,
  if necessary), and we assume $\gamma > 2 C_1$. Now the sum is
  comparable to the first term. Since $M(\log_2 M+2)\le C\gamma
  N^2\log^2 N$, we are done.

  The proof of the second part of the lemma is similar, and we omit it.
\end{proof}

\begin{proof}[Proof of theorem \ref{T:Z2_upper}]
  By lemmas \ref{L:Z2_delta_small} and \ref{L:general_upper}, in the DLA
  generated by the $\Z^2$ restricted walk, for all $\gamma$,
  \[
    \P(D_N \ge \gamma N^2 \log^2 N) \lesssim \frac{1}{\gamma}.
  \]

  Now take $N=2^k$ and $\gamma = k^{1+\eps}$ for some $\eps>0$. By
  Borel-Cantelli only a finite number of these events happen, and we get
  that $D_{N}<N^2\log^{3+\eps}N$ for $N=2^k$ sufficiently large. The bound
  for other $N$'s follows by monotonicity of $D_N$. Since $\eps$ was
  arbitrary, the theorem follows.
\end{proof}

\section{Walks with finite third moment}
\label{sec:finite third}

As discussed in the introduction, despite the fact that all symmetric walks
with finite variance scale to Brownian motion, the growth rates of the
aggregates resulting from such walks vary. In this chapter and the next we
analyze this phenomenon. We first consider the simpler case, where $\xi$
has a finite third moment. In that case the behaviour is similar to the
behaviour when the walk has bounded steps, i.e., the diameter grows
linearly. The case of walks with finite variance and infinite third moment
is more complex and is dealt with in chapter~\ref{sec:finite variation}.

\begin{thm} \label{T:third moment}
  If $\E|\xi|^3<\infty$ and $\E\xi=0$ then there is some $C$ so that $\limsup
  \frac{D_n}{n} < C$ a.s.
\end{thm}

Thus in the case $\alpha>3$ (and sometimes when $\alpha=3$), the diameter
grows linearly. Of course, the diameter $D_n$ must be at least $n-1$, so
the theorem gives the correct rate of growth. This suggests that the
process behaves much as in the case where the jumps are bounded: only a few
particles in the extremes of $A_n$ affect subsequent growth, and the limit
aggregate will have some positive density. In part III we shall discuss
existence of the limit.

We start with a technical lemma regarding hitting probabilities of our
walks, which will be useful also in chapter \ref{sec:finite variation}.

\begin{lemma} \label{L:hit_x}
  Let $R$ be a random walk on $\Z$ with steps of mean 0 and variation
  $\sigma^2<\infty$. Then there are $c,C>0$ such that for any $A\subset\Z$,
  $A\neq\emptyset$,
  \begin{enumerate}
  \item\label{enu:limy>c} If $x>\max A$, then $\lim_{y\to\infty}
  \P_y(T_x<T_A) > c$.
  \item\label{enu:escp dist} If $\dist(x,A)$ is large enough then $c <
  \dist(x,A) \P_x(T_A<T_x) < C$.
  \end{enumerate}
\end{lemma}

Note that the limit in clause \ref{enu:limy>c} exists since the random walk
is recurrent and the harmonic measure on the set $A\cup\{x\}$ exists. If
$A$ lies to one side of $x$, then in clause \ref{enu:escp dist} $c$ and $C$
can be arbitrarily close (this follows from the proof below). The lemma is
related to the asymptotic linearity of the harmonic potential for the
random walk \cite[T28.1, P29.2]{S76}, but we chose a somewhat different
path for the proof.

\begin{proof}
  Let $T_-$ denote the hitting time of $\Z^-$. Define the half-line Green
  function
  \[
  g(y,x) = \sum_{n\ge 0} \P_y(R(n)=x, n<T_-),
  \]
  i.e., the mean time spent at $x$ before $T_-$. By \cite[P19.3, P18.8,
  T18.1]{S76}, $g$ has a representation using two auxiliary functions $u$
  and $v$,
  \[
  g(y,x) = \sum_{n=0}^{\min(x,y)} u(x-n) v(y-n);
  \]
  the limits $\displaystyle \lim_{n\to\infty} u(n)$ and $\displaystyle
  \lim_{n\to\infty} v(n)$ both exist, and their product is $\displaystyle
  \lim_{n\to\infty} u(n)v(n) = 2/\sigma^2$. It follows that for any
  $\ep>0$, there is an $x_0(\ep)$ such that for any $y\ge x \ge x_0$ we
  have
  \begin{equation} \label{eq:G_limit}
  \left|g(y,x) - \frac{2}{\sigma^2} x \right| \le \ep x.
  \end{equation}

  By the strong Markov property for the hitting time of $\Z^-\cup\{x\}$, we
  have
  \[
  g(y,x) = \delta_{y,x} + \P_y(T_x<T_-) g(x,x),
  \]
  where $\delta$ is the Kronecker delta function. In particular we find for
  $y\ge x\ge x_0$ that
  \[
  \P_y(T_x<T_-) = \frac{g(y,x)-\delta_{y,x}} {g(x,x)}.
  \]
  Applying \eqref{eq:G_limit} we find that
  \begin{equation} \label{eq:lim_inf}
    \inf_{y\ge x} \P_y(T_x<T_-) \xrightarrow[x\to\infty]{} 1,
  \end{equation}
  thus if $x$ is far from $\Z^-$ and the walk starts to the right of $x$,
  it is likely to hit $x$ before $\Z^-$. If the walk starts at $x$ we have
  the asymptotics
  \begin{equation} \label{eq:Pxx}
  \P_x(T_-<T_x) = \frac{1}{g(x,x)} = \frac{\sigma^2}{2x}(1+o(1))
  \end{equation}
  as $x\to\infty$.

  To prove \ref{enu:limy>c}, note that by translation we may assume $\max
  A=0$. By monotonicity in $A$, it suffices to prove the bound for
  $A=\Z^-$. In that case the limit is just $\lim_{y\to\infty}
  \P_y(T_x<T_-)$. This limit is positive for any $x$ (since the random walk
  is aperiodic), and by \eqref{eq:lim_inf} it is close to 1 for large $x$.
  In particular it is always greater than some $c$.

  \medskip

  To prove the upper bound in \ref{enu:escp dist}, note that by
  monotonicity in $A$, given $d=\dist(x,A)$ it suffices to prove the
  bound for $A=\Z\setminus(x-d,x+d)$. Using \eqref{eq:Pxx}, by
  translation we find that
  \[
  \P_x(T_{(-\infty,x-d]} < T_x) = \frac{\sigma^2}{2d}(1+o(1))
  \]
  as $d\to\infty$. By symmetry, $\P_x(T_{[x+d,\infty)} < T_x)$ has the same
  asymptotics. A union bound gives for any $A$
  \[
  \P_x(T_A<T_x) \le \frac{\sigma^2 + o(1)}{\dist(x,A)}.
  \]

  It remains to prove the lower bound of (2). By monotonicity in the set
  $A$, it suffices to prove the bound for $A$ that consist of a single
  point $y$, so that $d=\dist(x,A)=|x-y|$. We consider here only the case
  $y=x-d$ as the case of $y=x+d$ is symmetric. Define the interval
  $B=(-\infty,y]$. In order to hit $y$ before returning to $x$, the random
  walk must hit $B$ (possibly at $y$) before returning to $x$. We have
  \[
  \P_x(T_y<T_x) = \P_x(T_B<T_x) \cdot \P(T_y<T_x|T_B<T_x).
  \]
  The second term on the RHS is an average over starting points in $B$ of
  the probability that $y$ is hit before $x$. When $d$ is large, these
  probabilities are close to 1, uniformly over $B$, since
  \eqref{eq:lim_inf} estimates the probability of hitting $y$ before
  hitting $[x,\infty)$. Thus the weighted average is also close to 1. As
    $d\to\infty$, we find
  \[
  \P_x(T_y<T_x) \sim \P_x(T_B<T_x) \sim \frac{\sigma^2}{2d}.\qedhere
  \]
\end{proof}

\begin{proof}[Proof of \thmref{third moment}]
  Since $\P_{\pm\infty}(T_x<T_A)\le 1$ we find by Lemma~\ref{L:hit_x} and
  the gluing formula (\ref{eq:murecurr}) that
  \[
  \mu(x) = \sum_a\mu(x,a) \le \frac{p(x,A)}{c/\dist(x,A)} \le C \dist(x,A)
  \P(\xi>\dist(x,A)).
  \]
  Summing over all $x$ with $\dist(x,A)>t$ we get
  \[
    \P(\Delta D_n > t) \le C \sum_{k>t} k \P(\xi>k).
  \]
  Thus we have the stochastic domination $\Delta D_n \slt Y$ with
  \[
    \P(Y>t) = 1\wedge C \sum_{k>t} k \P(\xi>k).
  \]
  However,
  \begin{align*}
    \E Y = \sum_t \P(Y>t)
    & \le C \sum_{\substack{t,k\\t<k}} k \P(\xi>k) \\
    & \le C \sum_{k} k^2 \P(\xi>k) \le C\E|\xi|^3 < \infty.
  \end{align*}
  We find that $D_n$ is dominated by a sum of $n$ independent copies of
  $Y$. By the law of large numbers, $\limsup \frac{D_n}{n} \le \E Y <
  \infty$.
\end{proof}

\section{Walks with finite Variance}
\label{sec:finite variation}

In this chapter we will analyze random walks with $2<\alpha<3$ and show
that the aggregate grows like $n^{2/(\alpha-1)+o(1)}$. We prove the lower
bound in section \ref{sec:low23} and the upper bound in section
\ref{sec:high23}. The upper bound is the harder of the two.

We remark that the upper bound on the growth of the aggregate requires only
an upper bound on $\P(|\xi|>t)$, while the lower bound on the aggregate
requires a lower bound on $\P(|\xi|>t)$ and the assumption of a finite
second moment. Hence we get information also on the case that $\P(|\xi|>k)$
decays irregularly, and satisfies only $c
k^{\alpha_1}<\P(|\xi|>k)<Ck^{\alpha_2}$. Theorems~\ref{T:23_lower} and
\ref{T:23_upper} still apply, and give upper and lower bounds on the
diameter (in terms of $\alpha_1$ and $\alpha_2$ respectively). In such
cases, the bounds leave a polynomial gap, and it is reasonable to believe
that $\frac{\log D_n}{\log n}$ also fluctuates between the corresponding
$\beta$'s.

\subsection{Lower bound}\label{sec:low23}

\begin{thm}\label{T:23_lower}
  Assume $\E(|\xi|^2)< \infty$ and $\E\xi=0$. Fix $\alpha\in(2,3]$, and let
  $\beta = \beta(\alpha) = \frac{2}{\alpha-1}$.
  \begin{enumerate}
  \item If $\P(|\xi|>t) \ge ct^{-\alpha}$, then a.s.\ $\limsup n^{-\beta}
    D_n = \infty$, and $D_n \ge \frac{n^\beta}{\log \log n}$ for all large
    enough $n$.
  \item If one has only $\P(|\xi|>t) \ge t^{-\alpha+o(1)}$ for $2<\alpha<3$
    then a.s.\ $D_n \ge n^{\beta+o(1)}$.
  \end{enumerate}
\end{thm}

Let $D^+n = \max A_n$ and $D^-_n = -\min A_n$, so that $D_n = D^+_n+D^-_n$.
With subsequent papers in mind, we work with $D^\pm_n$ instead of $D_n$.
One small argument that is needed to work with $D^+_n$ is the following
lemma. Let $A_n^+ = A_n\cap \Z_+$, be the positive elements of $A_n$.

\begin{lemma}\label{L:many_positive}
  There exists $c_0 > 0$ such that a.s.\ $\liminf |A_n^+|/n > c_0$.
\end{lemma}

\begin{proof}[Proof of lemma]
  Fix $k>0$ such that $P(\xi = -k) > 0$. Fix $n$, and consider the
  probability of gluing $x_n = D^+_n + k$ to $A_n$. By \lemref{hit_x},
  $\P_\infty(T_{x_n} < T_{A_n}) > c$, and therefore $\mu(x_n;A_n) \geq c
  \P(\xi=-k) > c > 0$. Since each time this happens a positive point is
  added to $A_n^+$, we find that $|A_n^+|$ dominates a sum of $n$
  i.i.d.\ Bernoulli variables.
\end{proof}

\begin{proof} [Proof of \thmref{23_lower}]
  Consider some $x$ to the right of $A=A_n$. Since $\E|\xi|^2 < \infty$
  \lemref{hit_x} applies, giving the bounds $\P_{\pm\infty}(T_x<T_A) > c >
  0$ and $\P_x(T_A<T_x) \approx 1/\dist(x,A)$. Therefore
  \begin{equation} \label{eq:mu2mmnt}
    \mu(x,a) \gtrsim \dist(x,A) p_{x,a}.
  \end{equation}
  Summing over positive $x>D_n^++m$ and all $a\in A_n^+$ we get
  \begin{align*}
    \P(\Delta (\max A_n) > m|\F_n)
    &\gtrsim \sum_{\substack{x>a\ge 0\\ x>D_n^++m}} \dist(x,A_n^+) p_{x,a} \\
    &\ge m \sum_{\substack{x>a\ge 0\\ x-a > m+D^+_n}} p_{x,a} \\
    &= m |A^+_n| \cdot \P(\xi > m+D^+_n).
  \end{align*}
  It follows that on the event $D^+_n < m$ we have
  \[
    \P(\Delta D^+_n > m|\F_n) \gtrsim m |A_n^+| \cdot \P(\xi>2m).
  \]
  Consider the events
  \[
  E_n(m) = \{D^+_{n+1} \ge m\} \cup \{|A_n^+| < c_0 n\},
  \]
  where $c_0$ is from \lemref{many_positive}. If $A_n^+$ is small, or if
  $\max A_n> m$ then $E_n(m)$ occurs. Hence we have the uniform bound
  \[
    \P(E_n(m) | \F_n) \gtrsim m n \P(\xi>2m).
  \]
  Note that this bound is uniform in $\F_n$, and hence the events $E_n(m)$
  stochastically dominate independent events with the above probabilities.

  Applying the tail estimate of part (i), one finds
  \[
    \P(E_n(m) | \F_n) \gtrsim n m^{1-\alpha}.
  \]
  Taking $m=m(n)=a n^\beta$ yields the bound $\P(E_n(a n^\beta)|\F_n) \gtrsim
  n^{-1}$. By Borel-Cantelli, a.s.\ infinitely many of the $E_n(a n^\beta)$
  occur for any $a>0$. By \lemref{many_positive} $\{|A_n^+| > c_0 n\}$ for
  all but finitely many $n$. This implies the $\limsup$ bound on $D^+_n$.

  For the uniform lower bound part of (i), take $E_n = E_n
  \left(\frac{n^\beta}{\log\log n}\right)$ to find
  \[
  \P(E_n | \F_n) \ge \frac{c(\log\log n)^{\alpha-1}}{n},
  \]
  Consequently, the probability that $E_n$ fails to occur for all
  $n\in[N,2N)$ is at most
  \[
  \prod_{n=N}^{2N} \left( 1-\frac{c(\log\log n)^{\alpha-1}}{n} \right)
  \le e^{ -c (\log\log N)^{\alpha-1} } \ll \frac{1}{\log N}.
  \]
  Looking at an exponential scale $N=2^k$ one finds that a.s.\ only
  finitely many scales are bad. On this event $D^+_n \ge \frac{n^\beta}
  {\log \log n}$ for all large enough $n$.

  Given the weaker tail estimate $\P(X>t) \ge t^{-\alpha+o(1)}$, we have
  for any $\alpha'>\alpha$ for some $c$, $\P(X>t) \ge ct^{-\alpha'}$. Thus
  part (i) implies that a.s.\ eventually $D_n \ge \frac{n^{\beta'}}{\log
    n}$. As $\alpha'$ decreases to $\alpha$ we can get $\beta'$ close to
  $\beta$.
\end{proof}

\subsection{Infinite third moment: upper bound}\label{sec:high23}

\begin{thm} \label{T:23_upper}
  Fix $\alpha\in(2,3]$ and let $\beta = \frac{2}{\alpha-1}$. If the random
  walk is such that $\P(|\xi|>t) \le c t^{-\alpha}$ and $\E\xi=0$,
  then a.s.\ $D_n \le n^{\beta+o(1)}$.
\end{thm}

The proof below gives $D_n \lesssim n^\beta (\log n)^2$, and with minimal
modification one $\log n$ factor can be removed.

We analyze only $\max A_n$, noting that $\min A_n$ behaves identically.
This yields bounds on $D_N = \max A_N - \min A_N$. Let $J(n,m)$ be the
event that $\Delta (\max A_n) \ge m$. This will be referred to as ``making
a large jump to the right'' at time $n$. We treat $\max A_N$ as the sum of
all jumps made to the right. The key idea is that if many large jumps to
the right were already made, the probability of additional ones is smaller.
This analysis is carried out for multiple scales of jumps.

The crux of the proof is the following estimate.

\begin{lemma}\label{L:few_jumps}
  Assume $\P(\xi<-t) \le c t^{-\alpha}$ for some $\alpha\in(2,3)$, then
  \[
  \sum_{n\le N} \P(J(n,m)) \le C N m^{(1-\alpha)/2} = C N m^{-1/\beta}.
  \]
\end{lemma}

\begin{proof}
  Define for $a\in A_n$
  \[
    W(a) = W_n(a) = \max A_n - a
  \]
  to be the distance from $a$ to the rightmost point of $A_n$. Using
  \lemref{hit_x} and the gluing formula (\ref{eq:murecurr}) we have the bound
  \begin{align*}
    \P(J(n,m)|\F_n)
    &\lesssim \sum_{a\in A_n} \sum_{x\ge m+\max A_n} p_{x,a} \dist(x,A) \\
    &= \sum_{a\in A_n} \sum_{d\ge m} d \P(\xi = -W(a)-d)  \\
    &= \sum_{a\in A_n} \Big(m\P(\xi\le -m-W(a)) +
                       \sum_{d > m} \P(\xi\le -d-W(a))\Big)  \\
    &\lesssim \sum_{a\in A_n} (m+W(a))^{1-\alpha}.
  \end{align*}

  We proceed to use this bound to estimate the total expected number of
  such jumps up to time $N$. The idea is that if jumps are frequent then
  the maximum of $A_n$ quickly moves away from any fixed $a\in A_n$, and so
  $W(a)$ is large and the probability of additional jumps is small.

  Fix some $L$ (to be determined later) and define $\hat J(n,m)$ to be the
  event that $J(n,m)$ occurs and that either $n\le L$ or $J(n',m)$ occurs
  for some $n'\in[n-L,n)$. Thus $\hat J(n,m)$ denotes the event that there
  is a large jump at time $n$ and the process has waited at most $L$ steps
  since the previous large jump (or from the beginning).
  In particular, $J(n,m)$ and $\hat J(n,m)$ can differ only once in any
  $L$ consecutive $n$'s. Thus when $L$ is large, $\hat J(n,m)$ is typically
  the same as $J(n,m)$, and there are at most $\lfloor N/L\rfloor$
  different $n\leq N$ when $J$ occurs and $\hat J$ does not.

  Let $\{t_i\}$ be the set of times $n$ at which $J(n,m)$ occurs, including
  (for notational convenience) $t_0=0$ and $t_{k+1}=N$, where $k$ is the
  number of large jumps that occur. Let $s_i=t_i-t_{i-1}$ be the times
  spent waiting for large jumps. Finally, let $\hat s_i = \min(s_i,L)$.

  Consider a particle at position $a$ that has been added in the time
  interval $(t_i,t_{i+1}]$. At any later time $n\in (t_j,t_{j+1}]$ we
  have
  \[
    W_n(a) \ge (j-i) m
  \]
  since there have been at least $j-i$ large jumps to the right after the
  particle was added to the aggregate. We now have
  \begin{align*}
    \sum_{n\le N} \P(\hat J(n,m)|\F_n)
    &\lesssim \sum_{j=0}^k \sum_{n=t_j}^{t_j+\hat s_j} \sum_{a\in A_n}
         (m+W_n(a))^{1-\alpha} \\
    &\le \sum_{j=0}^k \sum_{n=t_j}^{t_j+\hat s_j} \sum_{i=0}^j s_i
         (m(1+j-i))^{1-\alpha} \\
    &= m^{1-\alpha} \sum_j \sum_{i=0}^j s_i \hat s_j (1+j-i)^{1-\alpha}\\
  \intertext{and since $\hat s_j\leq L$}
    &\le m^{1-\alpha}L\sum_i s_i \sum_{j\ge i} (1+j-i)^{1-\alpha}\\
    &\lesssim m^{1-\alpha} L \sum_i s_i = m^{1-\alpha} L N.
  \end{align*}
  We now integrate over $\F_n$ to get
  \[
    \sum\P(\hat J(n,m)) \le Cm^{1-\alpha}LN\,.
  \]
  Since the difference between $\sum\P(J)$ and $\sum\P(\hat J)$ is bounded
  by $\lfloor N/L \rfloor$, we get
  \[
    \sum \P(J(n,m)) \le \lfloor N/L \rfloor + C m^{1-\alpha} N L.
  \]
  Setting $L = m^{(\alpha-1)/2}$ completes the proof.
\end{proof}

\begin{proof}[Proof of \thmref{23_upper}]
  Given $n\leq N$, let $\ell=\log N$, and let $\tau_n$ be the sum of all
  jumps to the right of size at most $N^\beta \ell^2$ up to time $n$.
  (``$\tau$'' for ``truncated''). By \lemref{few_jumps} The probability
  that by time $N$ there is some jump to the right of size at least
  $N^\beta \ell^2$ is at most $C N (N^\beta \ell^2)^{-1/\beta} = C
  \ell^{-2/\beta}$. Considering a geometric sequence of $N$'s, since
  $\beta<2$, we find that $\max A_n = \tau_n$ for all large enough $N$, and
  all $n\leq N$.

  Truncating jumps at $N^\beta \ell^2$, we have that
  \begin{align*}
  &&\E \tau_n &=\sum_{n=0}^{N-1}\sum_{m=1}^{N^\beta\ell^2}m\P(\max
    A_{n+1}=\max A_n + m ) \\
  &\mbox{By Abel resummation}&&\leq
    \sum_{m=1}^{N^\beta \ell^2}\sum_{n\leq N} \P(J(n,m)) \\
  &\mbox{By \lemref{few_jumps}}&& \lesssim
    \sum_{m=1}^{N^\beta \ell^2} N m^{-1/\beta} \\
  &&& \lesssim N (N^\beta \ell^2)^{1-1/\beta} = N^\beta \ell^{2(1-1/\beta)}.
  \end{align*}
  By Markov's inequality, $\P(\tau_n>N^\beta \ell^2) < \ell^{-2/\beta}$.
  Considering a geometric sequence of $N$'s, we find that a.s.\ $\tau_n
  \leq c N^\beta \ell^2$ for all large enough $N$.
\end{proof}

\begin{rem*}
  Suppose one tries to prove \thmref{23_upper} like before, i.e.\ like
  \lemref{general_upper}, \thmref{third moment} or
  \clmref{alpha12_delta_small} below. In other words, one looks for
  \emph{uniform} estimates for $\P(D_N>m\,|\,\F_n)$. The best we could find
  was
  \[
  \P(D_N>m\,|\,\F_n)\stackrel{\eqref{eq:mu2mmnt}}{\lesssim}
  \min(nm^{1-\alpha},m^{2-\alpha}).
  \]
  This only gives an upper bound of $D_n\le n^{4-\alpha}$ which is not
  sharp at any $\alpha\in(2,3)$. The failure of this uniform estimate
  approach means that one must use some information about the structure of
  the aggregate. However, the proof of \lemref{few_jumps} demonstrates that
  we do not need to know too much about the structure of $A_n$ --- only
  that it is not too concentrated near its right (or left) extremal points.
\end{rem*}

\section{Walks with infinite variance}
\label{sec:12}

\subsection{Preliminaries}

Walks with $\alpha\in(1,2)$ all fall into this category. Any walk with mean
0 is recurrent \cite[P2.8]{S76}, and in particular so is any symmetric walk
with finite mean. Thus we can use the gluing formula \eqref{eq:murecurr} to
calculate gluing probabilities. At the moment our techniques do not work
for completely general walks in this regime, but only for walks with
sufficiently nice tail behaviour. Specifically, we focus on walks that are
in the domain of attraction of a stable process. In particular, our results
apply to any walk with $\P(\xi>t) = (c+o(1)) t^{-\alpha}$. Our main result
here is the following.

\begin{thm}\label{T:12_both}
  If $\xi$ is a symmetric variable satisfying $\P(\xi>t)\asymp
  t^{-\alpha}$ with $1<\alpha<2$ then
  a.s.\ $D_n=n^{2+o(1)}$.
\end{thm}

Recall the definition of a slowly varying function

\begin{defn}
  A function $h$ is {\em slowly varying} at 0 (resp.\ at $\infty$) if for
  any $x>0$,
  \[
    \lim_{t\to0} \frac{h(tx)}{h(t)} = 1
  \]
  (resp.\ $\lim_{t\to \infty}$) and the limit is uniform on any compact set
  of $x$'s. For functions $f$, $g$ we write $f\asymp g$ if $f/g$ is slowly
  varying.
\end{defn}

Note that a common definition of slowly varying (see e.g.\ \cite{IL71})
requires only that the limit exists for all $x$. Since this is almost
impossible to use, one then applies Karamata's theorem \cite[Appendix
1]{IL71} to show that any locally integrable function which is slowly
varying in the weaker sense, is also slowly varying in the stronger
(uniform) sense stated above. Occasionally when we quote results from
\cite{IL71} we implicitly use Karamata's theorem to translate from the
weaker to the stronger sense of slowly varying.

A simple consequence of the definition of a slowly varying function at 0 is
that for any $\ep$ there are $K,\delta$ so that for $x<y<\ep$
\[
  \frac1K \left(\frac{x}{y}\right)^\delta < \frac{h(x)}{h(y)} <
  K\left(\frac{y}{x}\right)^\delta
\]
and $K \to 1$ and $\delta \to 0$ as $\eps \to 0$. If the function is slowly
varying at $\infty$, the same bounds hold for $y > x > \eps^{-1}$ instead.

Following \cite{S76,LL-book}, we define the harmonic potential by
\[
  a(n) = \sum_t \Big( \P_0(R_t=0) - \P_0(R_t=n) \Big).
\]
The harmonic potential is closely related to the Green function, and
the first stage is establishing its asymptotics. \cite[T28.1]{S76}
ensures us that the sum indeed converges.

\begin{lemma} \label{L:stable}
  Assume $\P(\xi>n) \asymp n^{-\alpha}$ for some $1<\alpha<2$. Then the
  harmonic potential satisfies
  \[
  a(n) \asymp n^{\alpha-1}.
  \]
\end{lemma}

\begin{proof}
  Given the tail of the step distribution we know from
  \cite[Theorem~2.6.1]{IL71} that $\xi$ belongs to the domain of attraction
  of a (symmetric) stable random variable with exponent $\alpha$. Denote by
  $\phi(\zeta)$ the Fourier transform of $\xi$. By
  \cite[Theorem~2.6.5]{IL71}, we
  have as $\zeta \to 0$ for some real $\beta,\gamma$
  \[
  \log\phi(\zeta)-i\gamma\zeta  \asymp  -|\zeta|^{\alpha}
  \left( 1 - i\beta\frac{\zeta}{|\zeta|}\tan(\tfrac\pi2\alpha) \right)
  \]
  ($\beta$ is the skewness of the stable limit, $\gamma$ corresponds to
  drift). In our case, $\xi$ is symmetric so $\phi$ is real valued and so
  $\beta=\gamma=0$ and
  \[
    \log \phi (\zeta) \asymp -|\zeta|^\alpha.
  \]
  This is the most essential use we make of the symmetry of $\xi$. In
  effect if one only assumes that the drift $\gamma$ is zero then the proof
  of the lemma follows through. However, for $\gamma\neq 0$ the
  conclusion of the lemma does not hold.

  Now, write
  \[
  \sum_{t\le T} \Big( \P_0(R_t=0) - \P_0(R_t=n) \Big) =
  \int_{-\pi}^{\pi} \big(1-e^{in\zeta}\big) \sum_{t\le T} \phi^{t}(\zeta)
  d\zeta.
  \]
  Aperiodicity gives that $\phi=1$ only at $\zeta=0$ and therefore (since
  $\alpha<2$), $\frac{1-e^{in\zeta}}{1-\phi(\zeta)}$ is integrable. Hence by
  dominated convergence,
  \[
    a(n) = \int_{-\pi}^{\pi} \frac{1-e^{in\zeta}}{1-\phi(\zeta)} d\zeta
    = 2{\rm Re} \int_0^\pi \frac{1-e^{in\zeta}}{|\zeta|^\alpha} h(\zeta)
    d\zeta
  \]
  where $h$ is slowly varying at $0$. Our plan is to use the fact that $h$
  is slowly varying and that bulk of the contribution to the last integral
  comes from $\zeta\in[\ep/n,1/(\ep n)]$ to compare this integral to
  $K_\alpha n^{\alpha-1} h(n^{-1})$. We begin with the constant and work
  backwards towards $a(n)$.

  We begin with
  \[
  \int_\ep^{1/\ep} \frac{1-e^{i x}}{x^\alpha} d x = K_\alpha + \eta_1(\ep),
  \]
  where $K_\alpha$ is the integral from $0$ to $\infty$ and where
  $\eta_1(\ep) \xrightarrow[\ep\to0]{}0$ (since $1<\alpha<2$). This
  integral may be calculated explicitly. For example, one may change the
  path of integration to the imaginary line (so that $e^{ix}$ is
  transformed into $e^{-x}$) and integrate by parts to get the integral
  defining the Gamma function. The result is that $K_\alpha =
  \Gamma(1-\alpha) e^{i\pi(1-\alpha)/2}$ and in particular is nonzero.

  Since $h$ is slowly varying, using the compactness of the interval
  $[\ep,1/\ep]$:
  \[
    \int_\ep^{1/\ep} \frac{1-e^{ix}}{x^\alpha}
                     \frac{h(xn^{-1})}{h(n^{-1})} dx
    = K_\alpha + \eta_1(\ep) + \eta_2(\ep,n),
  \]
  where for any fixed $\ep$ we have $\eta_2(\ep,n)
  \xrightarrow[n\to\infty]{}0$.

  On the interval $[0,\ep]$ we have
  \[
    \left|\frac{h(xn^{-1})}{h(n^{-1})}\right| < C x^{-\delta}
  \]
  where $\delta$ can be made arbitrarily small as $\ep\to0$ uniformly
  in $n$. Since
  $\alpha<2$ we have
  \[
    \left| \int_0^\ep \frac{1-e^{ix}}{x^\alpha}
                      \frac{h(xn^{-1})}{h(n^{-1})} dx \right|
    \le \int_0^\ep C x^{1-\alpha-\delta} \le C' \ep^{2-\alpha-\delta}.
  \]
  Similarly, on the interval $[1/\ep,\ep n]$ we have
  \[
    \left|\frac{h(xn^{-1})}{h(n^{-1})}\right| < C x^\delta
  \]
  where $\delta$ can be made small provided $n^{-1}$ and $x n^{-1} \leq
  \eps$ are both small. Since $\alpha<2$ we have
  \[
  \left| \int_{1/\ep}^{\ep n} \frac{1-e^{ix}}{x^\alpha}
                              \frac{h(xn^{-1})}{h(n^{-1})} dx \right|
  \le \int_{1/\ep}^\infty 2C x^{-\alpha+\delta} 
  \le C' \ep^{\alpha-1-\delta}.
  \]
  Since $1<\alpha<2$, both these bounds vanish as $\ep\to0$, and so we get
  \[
    \int_0^{\ep n} \frac{1-e^{ix}}{x^\alpha}
                   \frac{h(xn^{-1})}{h(n^{-1})} dx
    = K_\alpha + \eta_1(\ep) + \eta_2(\ep,n) + \eta_3(\ep,n),
  \]
  where $\limsup_{n\to\infty} |\eta_3(\ep,n)| \xrightarrow[\eps\to0]{} 0$.
  (Since for the $\limsup$ it suffices to consider $n>1/\eps$.)

  Now we are ready to consider $s(n)$. By a change of variable
  \[
    \int_0^\ep \frac{1-e^{in\zeta}}{|\zeta|^\alpha} h(\zeta) d\zeta
    = n^{\alpha-1} \int_0^{\ep n} \frac{1-e^{ix}}{x^\alpha} h(xn^{-1}) d\zeta
  \]
  Finally,
  \[
  \left|\int_\ep^\pi \frac{1-e^{in\zeta}}{|\zeta|^\alpha} h(\zeta) d\zeta
  \right|
  \le \int_\ep^\pi \frac{2}{|\zeta|^\alpha} h(\zeta) d\zeta,
  \]
  which is finite and independent of $n$.

  Combining these identities we get
  \[
  a(n) = n^{\alpha-1} h(n^{-1})
           \left[K_\alpha + \eta_1 + \eta_2 + \eta_3 \right] + \eta_4(\ep),
  \]
  with $\eta_4(\eps)$ bounded. Using the estimates on the $\eta_i$'s we
  find
  \[
  \lim_{\eps\to0} \limsup_{n\to\infty} \left| \frac{a(n)}{n^{\alpha-1}
      h(n^{-1})} - K_\alpha \right|
  = \lim_{\eps\to0} \limsup_{n\to\infty} \left| \eta_1 + \eta_2 + \eta_3 +
    \frac{\eta_4}{n^{\alpha-1} h(n^{-1})} \right|
  = 0.
  \]
  Since $a(n)$ does not depend on $\eps$, this in fact means that
  \[
  \limsup_{n\to\infty}\left| \frac{a(n)}{n^{\alpha-1}
      h(n^{-1})} - K_\alpha \right|=0
  \]
  and since $K_\alpha \neq 0$, this means that $a(n)\asymp n^{\alpha-1}$.
\end{proof}

Asymptotics of the potential kernel allow us to derive the following two
estimates for the hitting probabilities.

\begin{lemma}\label{L:alpha12_tools}
  Assume $\P(\xi>n) \asymp n^{-\alpha}$ for some $1<\alpha<2$, and let $I$
  be the interval $[-n,0]$, and $k\in[n,2n]$. Then
  \begin{gather}
    \P_k(T_I < T_k) \approx n^{1-\alpha} h(n)  \label{eq:Tkk T0n}\\
    \P_\infty(T_k < T_I) > c,                \label{eq:Tkinf T0n}
  \end{gather}
  for some $c>0$ and slowly varying function $h$.
\end{lemma}

\begin{proof}
  Let $g(x,y)=g_I(x,y)$ be the Green function with respect to $I$, namely
  \[
  g(x,y) = \sum_{t=0}^{\infty} \P_x(R(t)=y ; T_I>t),
  \]
  and let $H(x,\cdot)=H_I(x,\cdot)$ be the hitting measures on $I$, namely
  \[
  H(x,i) = \P_x(R(T_I)=i).
  \]
  Finally, let $H(\infty,i)=\lim H(x,i)$ be the harmonic measure on $I$.

  We use \lemref{stable} and get that
  \begin{equation}
    a(n) = n^{\alpha-1}h(n),     \label{eq:defhn}
  \end{equation}
  where $h$ is a slowly varying function. In particular $a(n)$ is
  unbounded so $R$ is recurrent (this can also be inferred directly
  from \cite[P2.8]{S76}). Hence Theorem~T30.2 of \cite{S76} applies
  to $R$ (the condition that the walk is not ``left- or
  right-continuous'', to use Spitzer's terminology, is satisfied
  because $R$ is symmetric). Combining (c) and (d) of that theorem we
  get for every $x,y\in\Z\setminus I$
  \begin{equation}\label{eq:hashpitz}
    g(x,y) = -a(y-x) + \kappa + \sum_{i\in I} H(\infty,i) a(x-i)
             + \sum_{i\in I} H(x,i) a(y-i),
  \end{equation}
  where $\kappa=\kappa_I$ is some number. As a first step to understanding
  \eqref{eq:hashpitz}, let $y\to\infty$. Since $\xi$ has infinite
  second moment we may apply \cite[T29.1(1)]{S76} which states that
  \[
  \lim_{|y|\to\infty} a(y-x) - a(y) = 0    \qquad\forall x
  \]
  and hence
  \[
  \lim_{|y|\to\infty} -a(y-x) + \sum_{i\in I} H(x,i)a(y-i) = 0
                                                      \qquad \forall x.
  \]
  or
  \begin{equation}\label{eq:gxinf}
  g(x,\infty) := \lim_{|y|\to\infty} g(x,y)
               = \kappa + \sum_{i\in I} H(\infty,i) a(x-i).
  \end{equation}
  Setting $x=1$ we get
  \[
  \kappa + \sum_{i\in I} H(\infty,i) a(1-i) \geq 0
  \]
  (this is not obvious because $\kappa$ is a negative constant which is
  difficult to estimate directly from its definition). Consequently, for
  $x=k\in[n,2n]$ we get
  \begin{align}
    g(k,\infty) & \geq\sum_{i=-n}^{0}(a(k-i)-a(1-i))H(\infty,i) \nonumber \\
    & \geq\min_{i=-n,\dotsc,0}a(k-i)-a(1-i) \nonumber \\
    & \geq n^{\alpha-1}h(n)(2^{\alpha-1}-1)(1+o(1)). \label{eq:gkinf}
  \end{align}
  The last inequality requires clarification. Roughly, the minimum in
  the LHS is achieved
  when $i=-n$ and $k=n$. Other $i\in I$ and other $k\in[n,2n]$ give larger
  values. This involves some simple playing around with the definition of a
  slowly varying function, in the spirit of the previous lemma which we
  shall omit.

  On the other hand, it is easy to see that $g(k,\infty)/g(k,k)$ is the
  harmonic measure of $k$ in the set $I\cup\{k\}$. Because the walk is
  symmetric and this set has more than 1 point the harmonic measure of
  any point is at most 1/2, and hence $g(k,\infty) \leq \frac{1}{2}g(k,k)$.
  With \eqref{eq:gkinf} this implies
  \[
  g(k,k) \geq cn^{\alpha-1}h(n).
  \]
  However, we can also write
  \[
  g(k,k) \stackrel{\textrm{(\ref{eq:hashpitz},\ref{eq:gxinf})}}{=}
  g(k,\infty) + \sum_{i\in I} H(k,i)a(k-i)
         \le \tfrac12 g(k,k) + \sum_{i\in I} H(k,i)a(k-i)
  \]
  or
  \begin{align}
    g(k,k) &\leq 2\sum_{i\in I} H(k,i)a(k-i)
            \leq 2 \max_{i\in I} a(k-i)      \nonumber \\
           &\leq 2(3n)^{\alpha-1}h(n)(1+o(1)).  \label{eq:gkk} \\
           &\leq Cn^{\alpha-1}h(n). \nonumber
  \end{align}
  Since we have both upper and lower bounds we find, $g(k,k) \approx
  n^{\alpha-1}h(n)$. This implies our first goal \eqref{eq:Tkk T0n} since
  $\P_k(T_k<T_I) = g(k,k)^{-1}$.

  Similarly we get \eqref{eq:Tkinf T0n} from
  \[
  \P_\infty(T_k<T_I) = \frac{g(k,\infty)}{g(k,k)}
  \geq \frac{2^{\alpha-1}-1}{2\cdot3^{\alpha-1}}(1+o(1))\,.   \qedhere
  \]
\end{proof}

While \lemref{alpha12_tools} talks about hitting an interval, by
translation invariance and by monotonicity of $T_A$ in $A$, it implies
similar bounds for any set $A$ and sufficiently far point $x$.

\begin{coro}\label{C:a12_hit}
  Assume $\P(\xi>n) \asymp n^{-\alpha}$ for some $1<\alpha<2$, and let
  $x,A$ satisfy $x \ge \max A + \diam A$. Then
  \begin{gather}
    \P_x(T_A < T_x) \le \dist(x,A)^{1-\alpha} h(\dist(x,A))  
                                                     \label{eq:Tkk T0n2}\\
    \P_\infty(T_x < T_A) > c,                        \label{eq:Tkinf T0n2}
  \end{gather}
  for some $c>0$ and slowly varying function $h$.
\end{coro}

\subsection{Proof of the lower bound}

We begin by proving a uniform lower bound on the probability of making a
large jump. Its use is analogous to the role \lemref{Z2_delta_large} plays
in the restricted $\Z^2$ case.

\begin{lemma}
  Under the assumptions of \thmref{12_both}, uniformly in $m\ge n$,
  \[
    \P(D_{n+1} \ge m \;|\; \F_n) \ge \frac{n}{m^{1+o(1)}}.
  \]
\end{lemma}

\begin{proof}
  On the event $D_n \ge m$ the claim is trivial, so assume $D_n < m$.
  Consider some $x$ such that $\dist(x,A)>m$. By \corref{a12_hit}, for some
  slowly varying function $h$
  \[
  \P_x(T_A < T_x) \leq \dist(x,A)^{1-\alpha} h(\dist(x,A))
                  \leq C\dist(x,A)^{1-\alpha+\eps}
                  \leq C m^{1-\alpha+\eps}.
  \]
  (Since any slowly varying positive function $h$ has $h(k) \lesssim
  k^{\eps}$ for any $\eps>0$ --- the constant $C$ and all constants
  below may depend on $\eps$).

  By the second part of \corref{a12_hit}, if $m \geq D_n$ then
  $\P_\infty(T_x<T_A)>c$. Using these bounds in the gluing formula
  (\ref{eq:murecurr}), we find that for any $x$ with $\dist(x,A) \ge m$
  \[
  \mu(x,a) = \frac{p_{x,a} \P_\infty(T_x<T_A)}{\P_x(T_A < T_x)}
  \ge c p_{x,a} m^{\alpha-1-\ep}.
  \]
  Summing over all $a\in A$ and $x$ with $\dist(x,A)\geq m$ we get
  \begin{align*}
    \P(\Delta D_n \ge m\;|\;\F_n)
    &\ge \sum_a \sum_{x\ge a+m+D_n} \mu(x,a) \\
    &\ge c m^{\alpha-1-\eps} \sum_a \sum_{x\ge a+m+D_n} p_{x,a} \\
    &= c n m^{\alpha-1-\eps} \P(\xi \ge m+D_n).
  \end{align*}
  It follows that on the event $D_n \leq m$ we have
  \begin{align}
    \P(\Delta D_n \ge m|\F_n)
    &\geq c n m^{\alpha-1-\eps} \P(\xi\ge 2m)\nonumber\\
    &\geq c n m^{\alpha-1-\eps} m^{-\alpha-\eps} = c n m^{-1-2\eps}
                                               \label{eq:23_delta_large}
  \end{align}
  Since $\eps$ was arbitrary, the proof of the lemma is complete.
\end{proof}

\begin{proof}[Proof of \thmref{12_both} (lower bound)]
  This follows from the last lemma and \lemref{general_lower}.
\end{proof}

\subsection{Proof of the upper bound}

Once again, we first prove a uniform bound on the probability of making a
large jump. The theorem then follows from the bound the same way the upper
bound for the $\Z^2$ case (\thmref{Z2_upper}) follows from
\lemref{Z2_delta_small}.

\begin{claim}\label{C:alpha12_delta_small}
  Under the conditions of \thmref{12_both},
  \[
    \P(\Delta D_n>m |\F_n) \le \frac{n}{m^{1-o(1)}}.
  \]
\end{claim}

\begin{proof}
  Set $A=A_n$, and consider $x$ outside the convex hull of $A$ (so that its
  addition will increase the diameter. We have $\P_x(T_A<T_x) \geq \P_x(T_y
  < T_x)$, where $y$ is an arbitrary point in $A$. By \cite[P11.5 or
  T30.2]{S76} we have $\P_x(T_y<T_x) = (2a(x-y))^{-1}$. Using this in the
  gluing formula \eqref{eq:murecurr} gives
  \begin{align*}
    \P(\Delta D_n>m|\F_n)
    & = \sum_{y\in A}\sum_{x:\dist(x,A)>m}
                \frac{p_{x,y} \P_\infty(T_x<T_A)}{\P_x(T_A<T_x)} \\
    & \lesssim \sum_{y\in A} \sum_{|x-y|>m}
                p_{x,y} a(x-y).
  \end{align*}
  Now, by \lemref{stable}, $a(x-y) \leq |x-y|^{\alpha-1} h(|x-y|)$ for some
  slowly varying $h$, hence $a(x-y) \leq C|x-y|^{\alpha-1+\eps}$, for some
  $C=C(\eps)$. This yields
  \begin{align}
    \P(\Delta D_n>m|\F_n)
    & \lesssim \sum_{y\in A} \sum_{|x-y|>m} p_{x,y} |x-y|^{\alpha-1+\eps}
    \nonumber \\
    & \lesssim n \sum_{k>m} k^{\alpha-1+\eps} \P(\xi=k).  \label{eq:a12_ub}
  \end{align}
  (All $y$'s give the same contribution, and there are two $x$'s at any
  distance $k$.) To estimate the last sum, we use the following Abel-type
  summation formula: Suppose $\{a_n\},\{b_n\}$ are two sequences, such that
  $\{b_n\}$ is summable, and $a_n B_{n+1} \to 0$, where
  $B_s=\sum_{k=s}^\infty b_k$, then
  \[
  \sum_{n\ge m} a_n b_n = a_m B_m + \sum_{n>m} (a_n-a_{n-1}) B_n.
  \]
  (Restricting the sums to $n\leq N$ gives a discrepancy of $a_N B_{N+1}$,
  so if one series converges so does the other and the identity holds.)
  Setting $a_n = n^{\alpha-1+\eps}$ and $b_n=\P(\xi=n)$, we get
  \begin{align*}
    \sum_{k\ge m} k^{\alpha-1+\eps} \P(\xi=k)
    & = m^{\alpha-1+\eps} \P(\xi \ge m) + \sum_{k>m}
    \big(k^{\alpha-1+\eps}-(k-1)^{\alpha-1+\eps}\big) \P(\xi \ge k) \\
    &\leq m^{\alpha-1+\eps} \P(\xi \geq m)
     + C \sum_{k>m} k^{\alpha-2+\eps} \P(\xi \geq k) \\
    &\leq C m^{\alpha-1+\eps} m^{-\alpha+\eps}
     + \sum_{k>m} ck^{\alpha-2+\eps} k^{-\alpha+\eps} \\
    &\leq Cm^{2\eps-1}.
  \end{align*}
  The penultimate inequality follows from the conditions on $\xi$ together
  with the fact that a slowly varying function grows slower than any power.
  Combining this with \eqref{eq:a12_ub}, we get
  \[
  \P(\Delta D_n>m | \F_n) \leq \frac{Cn}{m^{1-2\eps}}
  \]
  Since $\eps$ is arbitrary, the proof is complete.
\end{proof}

\begin{proof}[Proof of \thmref{12_both} (upper bound)]
  By \clmref{alpha12_delta_small}, $\P(\Delta D_n>m |\F_n) \le
  nm^{-1-o(1)}$. By the second part of \lemref{general_upper}, this implies
  that $\P(D_n\ge \gamma \phi(n))\lesssim 1/\gamma$ for any $n$ and
  $\gamma$, with $\phi(n)=n^{2+o(1)}$. Set $\gamma = \log^{2} n$ and
  consider only the geometric sequence $n=2^k$. It follows by
  Borel-Cantelli that a.s.\ $D_n \le \phi(n)\log^2 n=n^{2+o(1)}$ for all
  large enough $n=2^k$, and by monotonicity this holds for any $n$, as
  needed.
\end{proof}

\section{Hyper-transient: the restricted $\Z^3$ walk}
\label{sec:Z3}

We consider here the restriction to $\Z$ of the simple random walk
on $\Z^3$, where $\Z$ is embedded in $\Z^3$, say as $\Z\times
\{(0,0)\}$. Because simple random walk on $\Z^3$ is transient,
so is our induced process on $\Z$. This means that the gluing
formula (\ref{eq:murecurr}) is no longer valid, nor is our definition
of DLA. Let us therefore start by stating the analog of $\mu(x,a)$ in
the transient case. For a set $A$ and an element $x$ (possibly in $A$)
we define the escape probability $E_A(x)$ and the capacity $\capa(A)$ by
\[
E_A(x)=\P_x(T_A=\infty)\qquad\capa(A)=\sum_{a\in A}E_A(a)\,.
\]
Now define the transient gluing measure by
\begin{equation}\label{eq:transmu}
\mu(x,a)=\mu(x,a;A)=\frac{p_{x,a}E_A(x)}{\capa(A)}\,.
\end{equation}
In part II we explain why \eqref{eq:transmu} is the natural analog of
\eqref{eq:murecurr} in the transient settings, but for now we take it as
the definition (note that in part II, \eqref{eq:transmu} contains
$E_A^*(x)$, the escape probabilities for the reversed walk, but here our
walk is symmetric). With $\mu(x,a)$ defined the aggregate is defined
exactly as in Definition \ref{def:DLA}. We keep the notations of $R$ and
$\xi$ for the walk, $A_n$ for the aggregate and $D_n=\diam A_n$.

We now consider the specific case of the $\Z^3$ restricted walk. It turns
out that we only need the following property

\begin{defn}
  A random walk on $\Z$ is said to be {\em log-avoiding} if for some $c>0$
  and any finite $A\subset\Z$, and any $x$ we have $E_A(x) \ge
  \frac{c}{\log |A|}$.
\end{defn}

\begin{prop}
  The restricted $\Z^3$ random walk is log-avoiding.
\end{prop}

\begin{proof}
  Since $\Z$ is embedded in $\Z^3$, take a cylinder of radius $|A|^2$
  around $\Z$, and let $B$ be the vertex boundary of the cylinder. Since
  the random walk projected orthogonally to the embedded copy of $\Z$ is a
  two dimensional random walk, for any $x\in\Z$ we have $\P_x(T_B<T_\Z) \ge
  \frac{c}{\log |A|}$ (see e.g.~\cite[Lemma 9]{BKYY}).

  The probability in $\Z^3$ of ever hitting a point at distance $d$ is or
  order $c/d$ (see \cite[T26.1]{S76} or \cite[Theorem 4.3.1]{LL-book}).
  Thus for any point in $B$, the probability of ever hitting some point of
  $A$ is at most $c/|A|$. Combining the two,
  \[
  E_A(x) \ge \P_x(T_B<T_\Z) \P_B(T_A=\infty)
         \ge \frac{c}{\log |A|}(1-c/|A|) \ge \frac{c'}{\log |A|}\,.\qedhere
  \]
\end{proof}

With this property, we have super-exponential growth:

\begin{thm}
  If $R$ be a log-avoiding random walk, then a.s.\ for any $C$, $D_n > C^n$
  infinitely often.
\end{thm}

\begin{proof}
  For any transient walk $\capa(A_n)\le n$, and by log-avoidance $E_A(x) \ge
  \frac{c}{\log n}$ and putting this into \eqref{eq:transmu} gives
  \begin{equation}\label{eq:Z3_add_x}
    \mu(x,a) \ge \frac{cp_{x,a}}{n\log n},
  \end{equation}
  Now, gluing any $x$ with $|x-a|>m$ will imply $D_{n+1}>m$, and
  so since there are $n$ possible $a$'s,
  \[
    \P(D_{n+1} > m|\F_n) \ge n\frac{c\P(\xi>m)}{n\log n}\,.
  \]
  Next we note that for any log-avoiding random walk with step $\xi$
  we have $\P(\xi>m)\ge \frac{c}{\log m}$, since this is $\ge$
  probability of escaping the interval $[-m,m]$. Therefore
  \[
    \P(D_{n+1} > m|\F_n) \ge \frac{c}{\log m\log n}.
  \]
  Taking $m=C^n$ we see that a.s.\ $D_{n+1}>C^n$ for infinitely many $n$.
\end{proof}

In light of this very fast growth, the following result is somewhat
surprising.

\begin{thm}\label{T:fill_all}
  If $R$ be a log-avoiding aperiodic random walk, then a.s.\ $A_\infty =
  \Z$ (where $A_\infty = \bigcup A_n$.)
\end{thm}

\begin{proof}
  Fix some point $x\in \Z$ with $p_{0,x}>0$. Taking $a=0$ in
  \eqref{eq:Z3_add_x} we get
  \[
     \mu(x,0) \ge \frac{cp_{0,x}}{n\log n}.
  \]
  It follows that a.s.\ $x\in A_\infty$. If $p_{0,x}=0$ we use the
  aperiodicity of the walk to find $0=a_1,a_2,\dotsc,a_k=x$ such that
  $p_{a_i,a_j}>0$. Since the same argument works with
  any $a_i\in A_n$ in place of 0, we can show inductively that
  a.s.~all $a_i$ are in $A_\infty$, and in particular $x$.
\end{proof}

Let us conclude by a related conjecture

\begin{conj*}
  For any transient random walk on $\Z$, $\capa(A_n)=o(n)$ a.s.
\end{conj*}

Our basis for this conjecture is similar to the argument for
\thmref{fill_all}: If the capacity grows linearly, then $A_\infty=\Z$ which
(morally) implies that $A_n$ should have particles clumped into small
intervals. However, that contradicts the assumption of linear capacity.

If this conjecture holds then one may show that $D_n=o(n^{1/\alpha})$ for
any $0<\alpha<\frac{1}{2}$, so the aggregate does not grow in a precisely
polynomial fashion, but rather has some sub-polynomial correction.

\bibliographystyle{alpha}
\bibliography{DLA}

\begin{thebibliography}{AAK01}

\bibitem[AAK01]{AAK01}
Rami Atar, Siva Athreya, and Min Kang.
\newblock Ballistic deposition on a planar strip.
\newblock {\em Electron. Commun. Probab.}, 6:31--38, 2001.

\bibitem[AN86]{AN86}
Michael Aizenman and Charles~M. Newman.
\newblock Discontinuity of the percolation density in one-dimensional {$1/\vert
  x- y\vert \sp 2$} percolation models.
\newblock {\em Commun. Math. Phys.}, 107(4):611--647, 1986.

\bibitem[Bar93]{B93}
Martin~T. Barlow.
\newblock Fractals, and diffusion-limited aggregation.
\newblock {\em Bull. Sci. Math.}, 117(1):161--169, 1993.

\bibitem[BBY08]{BBY08}
Itai Benjamini, Noam Berger, and Ariel Yadin.
\newblock Long-range percolation mixing time.
\newblock {\em Comb. Probab. Comput.}, 17(4):487--494, 2008.

\bibitem[BH08]{BH08}
Itai Benjamini and Christopher Hoffman.
\newblock Exponential clogging time for a one dimensional {DLA}.
\newblock {\em J. of Statist. Phys.}, 131(6):1185--1188, 2008.

\bibitem[Bis04]{B04}
Marek Biskup.
\newblock On the scaling of the chemical distance in long-range percolation
  models.
\newblock {\em Ann. Probab.}, 32(4):2938--2977, 2004.

\bibitem[BKYY]{BKYY}
Itai Benjamini, Gady Kozma, Ariel Yadin, and Amir Yehudayoff.
\newblock Entropy of random walk range.
\newblock to appear in Ann. Inst. Henri Poincar\'e (B) prob. and stat.,
  arXiv:0903.3179.

\bibitem[BPP97]{BPP97}
Martin~T. Barlow, Robin Pemantle, and Edwin~A. Perkins.
\newblock Diffusion-limited aggregation on a tree.
\newblock {\em Probab. Theory Related Fields}, 107(1):1--60, 1997.

\bibitem[BY08]{BY08}
Itai Benjamini and Ariel Yadin.
\newblock Diffusion limited aggregation on a cylinder.
\newblock {\em Commun. Math. Phys.}, 279(1):187--223, 2008.

\bibitem[CM01]{CM01}
L.~Carleson and N.~Makarov.
\newblock Aggregation in the plane and {L}oewner's equation.
\newblock {\em Commun. Math. Phys.}, 216(3):583--607, 2001.

\bibitem[EW99]{EW}
Dorothea~M. Eberz-Wagner.
\newblock {\em Discrete Growth Models}.
\newblock PhD thesis, University of Washington, 1999.
\newblock arXiv:math/9908030.

\bibitem[HM]{HM}
H{\aa}kan Hedenmalm and Nikolai Makarov.
\newblock Quantum hele-shaw flow.
\newblock preprint, arxiv:math.PR/0411437.

\bibitem[IL71]{IL71}
I.~A. Ibragimov and Yu.~V. Linnik.
\newblock {\em Independent and stationary sequences of random variables}.
\newblock Wolters-Noordhoff Publishing, Groningen, 1971.
\newblock With a supplementary chapter by I. A. Ibragimov and V. V. Petrov,
  Translation from the Russian edited by J. F. C. Kingman.

\bibitem[IN88]{IN88}
John~Z. Imbrie and Charles~M. Newman.
\newblock An intermediate phase with slow decay of correlations in
  one-dimensional $1/|x-y|^2$ percolation, {Ising} and {Potts} models.
\newblock {\em Commun. Math. Phys.}, 118(2):303--336, 1988.

\bibitem[Kes87]{K87}
Harry Kesten.
\newblock How long are the arms in {DLA}?
\newblock {\em J. Phys. A}, 20(1):L29--L33, 1987.

\bibitem[Kes90]{K90}
Harry Kesten.
\newblock Upper bounds for the growth rate of {DLA}.
\newblock {\em Phys. A}, 168(1):529--535, 1990.

\bibitem[Kes91]{K91}
Harry Kesten.
\newblock Some caricatures of multiple contact diffusion-limited aggregation
  and the {$\eta$}-model.
\newblock In {\em Stochastic analysis (Durham, 1990)}, volume 167 of {\em
  London Math. Soc. Lecture Note Ser.}, pages 179--227. Cambridge Univ. Press,
  Cambridge, 1991.

\bibitem[KS08]{KS08}
Harry Kesten and Vladas Sidorvicius.
\newblock A problem in one-dimensional diffusion-limited aggregation {(DLA)}
  and positive recurrence of markov chains.
\newblock {\em Ann. Probab.}, 36(5):1838--1879, 2008.

\bibitem[Law95]{L95}
Gregory~F. Lawler.
\newblock Subdiffusive fluctuations for internal diffusion limited aggregation.
\newblock {\em Ann. Probab.}, 23(1):71--86, 1995.

\bibitem[LBG92]{LBG92}
Gregory~F. Lawler, Maury Bramson, and David Griffeath.
\newblock Internal diffusion limited aggregation.
\newblock {\em Ann. Probab.}, 20(4):2117--2140, 1992.

\bibitem[LL]{LL-book}
Gregory~F. Lawler and Vlada Limic.
\newblock Random walk: A modern introduction.
\newblock Book draft, available from
  \url{http://www.math.uchicago.edu/~lawler/books.html}.

\bibitem[LP08]{LP}
Lionel Levine and Yuval Peres.
\newblock Spherical asymptoticsfor the rotor-router model in $\mathbb{Z}^{d}$.
\newblock {\em Indiana Univ. Math. J.}, 57(1):431--450, 2008.
\newblock arXiv:0503251.

\bibitem[NP95]{NP95}
Charles~M. Newman and Marcelo S.~T. Piza.
\newblock Divergence of shape fluctuations in two dimensions.
\newblock {\em Ann. Probab.}, 23(3):977--1005, 1995.

\bibitem[NS86]{NS86}
Charles.~M. Newman and Lawrence~S. Schulman.
\newblock One-dimensional {$1/\vert j-i\vert \sp s$} percolation models: the
  existence of a transition for {$s\leq 2$}.
\newblock {\em Commun. Math. Phys.}, 104(4):547--571, 1986.

\bibitem[NT]{NT08}
James Norris and Amanda Turner.
\newblock Planar aggregation and the coalescing brownian flow.
\newblock preprint, arXiv:0810.0211.

\bibitem[Ric73]{R73}
Daniel Richardson.
\newblock Random growth in a tessellation.
\newblock {\em Proc. Cambridge Philos. Soc.}, 74:515--528, 1973.

\bibitem[Sch83]{S83}
Lawrence~S. Schulman.
\newblock Long range percolation in one dimension.
\newblock {\em J. Phys. A}, 16(17):L639--L641, 1983.

\bibitem[She]{Sh}
Eric Shellef.
\newblock {IDLA} on the supercritical percolation cluster.
\newblock preprint, arXiv:\linebreak[0]0806.4771.

\bibitem[Spi76]{S76}
Frank Spitzer.
\newblock {\em Principles of random walks}.
\newblock Springer-Verlag, New York, second edition, 1976.
\newblock Graduate Texts in Mathematics, Vol. 34.

\bibitem[WS83]{WS81}
T.~A. Witten and L.~M. Sander.
\newblock Diffusion-limited aggregation.
\newblock {\em Phys. Rev. B (3)}, 27(9):5686--5697, 1983.

\end{thebibliography}

\noindent
{\bf Gideon Amir} \\
Department of Mathematics, University of Toronto\\
Toronto ON, M5S 2E4, Canada \\
{\sc gidi.amir@gmail.com}\\

\noindent
{\bf Omer Angel} \\
Department of Mathematics, University of British Columbia\\
Vancouver BC, V6T 1Z2, Canada\\
{\sc angel@math.ubc.ca} \\

\noindent
{\bf Itai Benjamini, Gady Kozma} \\
Department of Mathematics, Weizmann Institute of Science \\
Rehovot 76100, Israel \\
{\sc itai.benjamini@gmail.com, gady.kozma@weizmann.ac.il}

\end{document}